\documentclass[11pt,a4paper]{article}
\usepackage[utf8]{inputenc}
\usepackage{epsfig,graphicx,graphics}
\usepackage{amsmath}
\usepackage{amssymb}

\usepackage{amsthm}
\usepackage{amsmath}
\usepackage{amsfonts}
\usepackage{geometry}
\usepackage{epsfig}

\usepackage{color}
\usepackage{amssymb}
\usepackage{graphicx}
\usepackage{amsmath,amsfonts,amsthm,amstext}
\usepackage{mathrsfs}
\usepackage{hyperref}

\usepackage{tikz}

\newtheorem{teo}{Theorem}[section]
\newtheorem{lema}{Lemma}[section]

\newtheorem{propo}{Proposition}[section]
\theoremstyle{definition}
\newtheorem{defi}{Definition}[section]

\newtheorem{rek}{Remark}[section]

\newcommand{\supp}{\operatorname{supp}}

\newcommand{\modd}{\operatorname{mod}}

\newcommand{\essinf}{\operatorname{ess\,inf}}

\newcommand{\esssup}{\operatorname{ess\,sup}}

\newcommand{\card}{\operatorname{card}}

\newcommand{\diam}{\operatorname{diam}}

\newcommand{\ve}{\varepsilon}
\newcommand{\N}{\mathbb{N}}
\newcommand{\R}{\mathbb{R}}
\newcommand{\Z}{\mathbb{Z}}
\newcommand{\M}{\mathcal{M}}

\newenvironment{proof1}[1][\textit{\textbf{Proof \em(Theorem~\ref{teocentral1})}}]{\textit{#1.} }{\hfill $\Box$}

\newenvironment{proof2}[1][\textit{\textbf{Proof \em(Theorem~\ref{teocentral2})}}]{\textit{#1.} }{\hfill $\Box$}

\newenvironment{proof3}[1][\textit{\textbf{Proof \em(Theorem~\ref{correlationconj})}}]{\textit{#1.} }{\hfill $\Box$}

\newenvironment{proof4}[1][\textit{\textbf{Proof \em(Theorem~\ref{packingconjC})}}]{\textit{#1.} }{\hfill $\Box$}

\begin{document}

\bigskip

\title {Generic dimensional and dynamical properties of invariant measures of full-shift systems over countable alphabets}
\date{}
\author{Silas L. Carvalho \thanks{} ~~and~ Alexander Condori \thanks{ Work  partially  supported  by Vicerrectorado de Investigación-UNSCH}
}
\maketitle

{ \small \noindent $^{*}\,$Instituto de Ciências Exatas, Universidade Federal de Minas Gerais. Av. Pres. Antônio Carlos 6627, Belo Horizonte-MG, 31270-901, Brasil. \\ {\em e-mail:
silas@mat.ufmg.br 

\em
{\small \noindent $^{\dag\,}$Departamento Académico de Matemática y Física, Universidad Nacional de San Cristóbal de Huamanga. Pq. Portal Independencia Nro.
57 U.V. Parque Sucre, Ayacucho 05001, Perú. \\ {\em e-mail:
alexander.condori@unsch.edu.pe 
\em
\normalsize
\maketitle
\begin{abstract}
\noindent{
  In this work, we are interested in characterizing typical (generic) dimensional properties of invariant measures associated with the full-shift system, $T$, in a product space whose alphabet is a countable set. More specifically, we show that the set of invariant measures with infinite packing dimension equal to infinity is a dense $G_\delta$ subset of $\mathcal{M}(T)$, the space of $T$-invariant measures endowed with the weak topology, where the alphabet $M$ is a countable Polish metric space. We also show that the set of invariant measures with upper $q$-generalized fractal dimension (with $q>1$) equal to infinity is a dense $G_\delta$ subset of $\mathcal{M}(T)$, where the alphabet $M$ is a countable compact metric space. This improves the results obtained by Carvalho and Condori in \cite{AS} and \cite{AS2}, respectively. Furthermore, we discuss the dynamical consequences of such results, regarding the upper recurrence rates and upper quantitative waiting time indicator for typical orbits, and how the fractal dimensions of invariant measures and such dynamical quantities behave under an $\alpha$-Hölder conjugation.
}
\noindent
\end{abstract}

\section{Introduction}

Let $(M,\rho)$ be a complete separable (Polish) metric space and let $S$ be its $\sigma$-algebra of Borel sets. Let $(X,\mathcal{B})$ be the bilateral product of a countable number of copies of $(M,S)$. 
Naturally, $\mathcal{B}$ coincides with the $\sigma$-algebra of the Borel sets in the product topology. Let $d$ be any metric in $X$ which is compatible with the product topology. Then, $(X,d)$ is also a Polish metric space. 

 One defines in $X$ the so-called full-shift operator, $T$, by the action
 \[Tx=y,\]
 where $x=(\ldots,x_{-n},\ldots,x_n,\ldots)$, $y=(\ldots,y_{-n},\ldots,y_n,\ldots)$, and for each $i\in\mathbb{Z}$, $y_i=x_{i-1}$. $T$ is clearly an homeomorphism of $X$ onto itself.

 We consider in this work three different settings:
 \begin{enumerate}
 \item  $d$ is any metric compatible with the product topology;
 \item $d$ is such that $T$ is a Lipschitz map; this is the case when $d:X\times X\rightarrow\mathbb{R}$ is given by the law 
\begin{equation}\label{metric1}
  d(x,y) =   \sum_{|n|\geq 0} \frac{1}{2^{|n|}}  \frac{\rho(x_n,y_n)}{1+\rho(x_n,y_n)},
\end{equation}
with $x,y\in X$;
\item $d$ is a sub-exponential metric of the form 
\begin{equation}\label{metric2}
  d(x,y) =   \sum_{|n|\geq 0} \min \left\{\frac{1}{a_{|n|}+1}, d(x_n,y_n)\right\},
  \end{equation}  
with $x,y\in X$, where $(a_n)$ is any monotone increasing sequence such that $\sum_{k\ge 0}\frac{1}{a_{k}+1}<\infty$ and, for each $\alpha>0$, $\lim_{k\to\infty}\frac{a_k}{e^{\alpha k}}=0$ (for instance, let for each $n\in\N\cup\{0\}$, $a_n=n^2$).
 \end{enumerate}
 
Naturally, the metrics defined in~\eqref{metric1} and~\eqref{metric2}  induce topologies in $X$ which also are compatible with the product topology. We mention explicitly along the text the metric that is used in each setting; if a specific metric is not mentioned, we then assume that we are in the first case.

Let $\mathcal{M}(T)$ be the space of all $T$-invariant probability measures, endowed with the weak topology (that is, the coarsest topology for which the net $\{\mu_\alpha\}$ converges to $\mu$ if, and only if, for each bounded and continuous function $f$, $\int fd\mu_\alpha\rightarrow \int fd\mu$). Since $X$ is Polish, $\mathcal{M}(T)$ is also a Polish metrizable space (see~\cite{Sigmundlibro}).

Given $\mu\in\M(T)$, the triple $(X,T,\mu)$ is called an $M$-valued discrete stationary stochastic process (see~\cite{Parthasarathy1961,Sigmund1971,Sigmund1974}; see also~\cite{FuLu} for a discussion of the role of such systems in the study of continuous self-maps over general metric spaces).

The study of generic properties (in Baire's sense; see Definition~\ref{genericprop}) of such $M$-valued discrete stationary stochastic processes goes back to the works of Parthasarathy~\cite{Parthasarathy1961} (regarding ergodicity) and Sigmund~\cite{Sigmund1971,Sigmund1974} (regarding positivity of the measure on open sets, zero entropy of the measure for $M=\mathbb{R}$).



The present authors have obtained in \cite{AS2,AS} some results regarding dimensional properties of invariant measures of full-shift dynamical systems over $X=\prod_{-\infty}^{+\infty}M$ in case the alphabet $M$ is a perfect Polish (or compact) metric space.

More specifically, it has been shown in~\cite{AS2} (respectively, in~\cite{AS}), among other results, that if $M$ is a perfect compact (respectively, Polish) metric space and if one endows $X$ with the metric defined in~\eqref{metric2} (respectively,~\eqref{metric1}), then for each $q\ge 1$, the set of invariant measures with upper $q$-generalized fractal dimensions (respectively, packing dimensions) equal to infinity is a residual subset of $\mathcal{M}(T)$.



The techniques presented in these papers are valid exclusively for the situation where $M$ is a perfect set. So, it remains an open problem to extend such results for full-shift systems over \textit{countable} alphabets (these results are false for full-shifts over finite alphabets, since such systems are expanding; see~\cite{AS3,Fathi1989} for details). Our main goal in this work is precisely to show that this is possible. The problem of possibly extending these results 
in case $M$ is a countable set is posed in both papers~\cite{AS2,AS} (see Remark~1.1 in~\cite{AS2} and Introduction in~\cite{AS}).


Naturally, by extending such results to full-shift systems over countable alphabets, and by presenting some sufficient conditions for a conjugation between (topological) dynamical systems to preserve positive dimensions, we hope to extend
the class of dynamical systems for which the sets of invariant measures with infinite packing and upper $q$-generalized dimensions ($q\ge 1$) are residual.

Based on the results obtained in~\cite{AS}, we also have something to say about recurrence rates of points and about quantitative waiting time indicators.

Thus, before we precisely state our results, some preparation is required.

\subsection{Preliminaries}





Here, we discuss some definitions, motivations and results regarding fractal dimensions of invariant measures and recurrence problems. For a more complete discussion, see~\cite{Pesin1997}.


In what follows, $(X,d)$ is an arbitrary metric space and $\mathcal{B}=\mathcal{B}(X)$ is its Borel $\sigma$-algebra. 

\begin{defi}[radius packing $\phi$-premeasure, \cite{Cutler1995}] Let $\emptyset\neq E\subset X$, and let $0 < \delta < 1$. A $\delta\textrm{-}$\emph{packing} of $E$ is a countable collection of disjoint closed balls $\{B(x_k,r_k)\}_k$ with centers $x_k\in E$ and radii satisfying  $0<r_k\leq \delta/2$, for each $k\in\N$. Given a measurable function $\phi$, the radius packing ($\phi,\delta$)\textrm{-}premeasure of $E$ is given by the law
\begin{eqnarray*}
P^{\phi}_{\delta}(E)=\sup\left\{\sum _{k=1}^{\infty} \phi(2r_k)\mid \{B(x_k,r_k)\}_k \mbox{ is a } \delta\textrm{-}\mbox{packing} \mbox{ of E}\right\}.
\end{eqnarray*}
By letting $\delta\to 0$, one gets the so-called \emph{radius packing} $\phi\textrm{-}$\emph{premeasure}
\begin{eqnarray*}
P^{\phi}_{0}(E)=\lim_{\delta \to 0} P^{\phi}_{\delta}(E).
\end{eqnarray*}
One sets $P^{\phi}_{\delta}(\emptyset)=P^{\phi}_{0}(\emptyset)=0$.
\end{defi}
It is easy to see that $P^{\phi}_{0}$ is non-negative and monotone. Moreover, $P^{\phi}_{0}$ generally fails to be countably sub-additive. One can,
however, build an outer measure from $P^{\phi}_{0}$ by applying Munroe's Method I construction, described  in \cite{Munroe,Rogers}. This leads to the following definition.

\begin{defi}[radius packing $\phi\textrm{-}$measure, \cite{Cutler1995}]
The radius packing $\phi\textrm{-}$measure of $E\subset X$ is defined to be 
\begin{eqnarray}
\label{pakmeasure}
P^{\phi}(E)=\inf\left\{ \sum_{k}P^{\phi}_{0}(E_k)\mid E\subset \bigcup_k E_k \right\}.
\end{eqnarray}
The infimum in (\ref{pakmeasure}) is taken over all countable coverings $\{E_k\}_k$ of $E$. It follows that $P$ is an outer measure on the subsets of $X$.
\end{defi}

In an analogous fashion, one may define the Hausdorff $\phi\textrm{-}$measure. 
\begin{defi}[Hausdorff $\phi\textrm{-}$measure, \cite{Cutler1995}]
For $E\subset X$, the outer measure $H^{\phi}(E)$ is defined by
\begin{eqnarray}
\label{hausmeasure}
H^{\phi}(E)=\lim_{\delta\to 0}\inf \left\{\sum_{k=1}^{\infty} \phi(\diam(E_k))\mid \{E_k\}_k \mbox{ is a } \delta\textrm{-}\mbox{covering} \mbox{ of } E \right\},
\end{eqnarray}
where a $\delta\textrm{-}$\emph{covering} of E  is any countable collection $\{E_k\}_k$ of subsets of $X$ such that, for each $k\in\N$, $E\subset \bigcap_k E_k$ and $\diam(E_k)\leq \delta$. If no such $\delta\textrm{-}$covering exists, one sets $H_{\phi}(E)=\inf \emptyset=\infty$.
\end{defi}

Of special interest is the situation where given $\alpha>0$, one sets $\phi(t)=t^{\alpha}$. In this case, one uses the notation $P^{\alpha}_0$, and refers to $P^{\alpha}_0(E)$ as the $\alpha\textrm{-}$packing premeasure of $E$. Similarly, one uses the notation $P^{\alpha}(E)$ for the packing $\alpha\textrm{-}$measure of $E$, and  $H^{\alpha}(E)$ for the $\alpha\textrm{-}$ Hausdorff measure of $E$.

\begin{defi}[Hausdorff and packing dimensions of a set, \cite{Cutler1995}]
Let $E\subset X$. One defines the Hausdorff dimension of $E$  to be the critical point
\begin{eqnarray*}
\dim_H(E)=\inf\{\alpha>0\mid h^{\alpha}(E)=0\};
\end{eqnarray*}
one defines the packing dimension of $E$ in the same fashion.
\end{defi}

We note that $\dim_H(X)$ or $\dim_P(X)$ may be infinite for some metric space $X$. 

\begin{defi}[lower and upper packing and Hausdorff dimensions of a measure,\cite{Mattila}]\label{HPdim}
  Let $\mu$ be a positive Borel measure on $(X,\mathcal{B})$. 
  The lower and upper packing and Hausdorff dimension of $\mu$ are defined, respectively, by 
\begin{eqnarray*}
\dim_{K}^-(\mu)&=&\inf\{\dim_{K}(E)\mid \mu( E)>0, ~ E\in \mathcal{B}\},\\
\dim_{K}^+(\mu)&=&\inf\{\dim_{K}(E)\mid \mu(X\setminus E)=0, ~ E\in \mathcal{B}\},
\end{eqnarray*}
where $K$ stands for $H$ (Hausdorff) or $P$ (packing). If $\dim_{K}^-(\mu)=\dim_{K}^+(\mu)$, one denotes the common value by $\dim_{K}(\mu)$. 
\end{defi}

Let $\mu$ be a positive finite Borel measure on $X$. One defines the upper and lower local dimensions of $\mu$ at $x\in X$ by
$$\overline{d}_{\mu}(x)=\limsup_{\ve\to 0}\frac{\log \mu(B(x,\ve))}{\log \ve} ~~\mbox{ and }~~ \underline{d}_{\mu}(x)=\liminf_{\ve\to 0}\frac{\log \mu(B(x,\ve))}{\log \ve},$$ 
if, for every $\ve> 0$, $\mu(B(x;\ve))>0$; if not, $\overline{\underline{d}}_\mu(x):=+\infty$.

The next result shows that the essential infimum of the lower (upper) local dimension of a probability measure equals its lower Hausdorff (packing) dimension, whereas the essential supremum of its lower (upper) local dimension equals its upper Hausdorff (packing) dimension; see Appendix in~\cite{AS} for its  proof. 

\begin{propo} 
\label{BGT} 
Let $\mu$ be a probability measure on $X$. Then,
\begin{eqnarray*}
\mu\textrm{-}\essinf \underline{d}_{\mu}(x)=\dim_H^-(\mu)\leq \mu\textrm{-}\esssup \underline{d}_{\mu}(x)= \dim_H^+(\mu), \\
 \mu\textrm{-}\essinf \overline{d}_{\mu}(x)=\dim_P^-(\mu)\leq \mu\textrm{-}\esssup \overline{d}_{\mu}(x)= \dim_P^+(\mu). 
\end{eqnarray*}
\end{propo}




As mentioned in \cite{AS}, the Hausdorff dimension of relatively simple sets can be very difficult to calculate. Furthermore, the notion of Hausdorff dimension is not fully adapted to the dynamics itself (e.g., if $Z$ is a periodic orbit, then its Hausdorff dimension is zero, regardless of whether the orbit is stable, unstable or neutral). 

Thus, in order to obtain relevant information about the dynamics of a (continuous) map $f:X\rightarrow X$ (where $ X $ is a measurable Borel space), one must consider not only the geometry of the measurable set $ Z \subset X $, but also the distribution of points in $ Z$ under $f$. That is, one should be interested in the frequency with which a given point $ x \in Z $ visits a fixed subset $ Y \subset Z$ under $f$. If $ \mu $ is an ergodic measure for which $ \mu(Y) > 0 $, for a typical point $ x \in Z $, the average number of visits is equal to $ \mu(Y) $. Therefore, the orbit distribution is completely determined by the measure $ \mu $. On the other hand, the measure $\mu$ is completely specified by the distribution of a typical orbit. In this direction, Grassberger, Procaccia and Hentschel \cite{Grassberger} introduced the so-called correlation dimension of a probability measure in an attempt to produce a characteristic of a dynamical system that captures information about the typical global behavior (with respect to a measure invariant) of the trajectories by observing only one of them.

The formal definition  is as follows (see~\cite{PesinT1995,Pesin1993,Pesin1997}): let $(X,r)$ be a complete and separable (Polish) metric space, and let $T:X\rightarrow X$ be a continuous mapping. Given $x\in X$, $\ve>0$  and $n\in\N$, one defines the correlation sum of order $q\in\N\setminus\{1\}$ (specified by the points $\{T^i(x)\}$, $i=1,\ldots,n$) by
\begin{equation*}
C_q(x,n,\ve)=\frac{1}{n^q}\,\card\,\{(i_1\cdots i_q)\in \{0,1,\cdots, n\}^q\mid r(T^{i_j}(x),T^{i_l}(x))\leq \ve ~ \mbox{ for any } ~ 0\leq j,l\leq q\},
\end{equation*}
where $\card A$ is the cardinality of the set $A$. Given $x\in X$, one defines (when the limit $n\to\infty$ exists) the quantities 
\begin{equation}\label{cordim}
 \underline{\alpha}_q(x)=  \frac{1}{q-1}\lim_{\overline{ \ve\to 0}}   \lim_{n\to \infty}\frac{\log C_q(x,n,\ve)}{\log(\ve)}, ~~  ~~  \overline{\alpha}_q(x)= \frac{1}{q-1}\overline{\lim_{\ve\to 0}} \lim_{n\to \infty}\frac{\log C_q(x,n,\ve)}{\log(\ve)},
\end{equation}
the so-called the \emph{lower} and the \emph{upper $q$-correlation dimensions at $x$}. If the limit $\ve\to 0$ exists, we denote it by $\alpha_q$, the so-called \emph{$q$-correlation dimension at $x$}. In this case, if $n$ is large and $\ve$ is small, one has the asymptotic relation
$$C_q(x,n,\ve) \thicksim\ve^{(q-1)\alpha_q}.$$
 
$C_q(x,n,\ve)$ gives an account of how the orbit of $x$, truncated at time $n$, ``folds'' into an $\ve$-neighborhood of itself; the larger $C_q(x,n,\ve)$, the more ``tight'' this truncated orbit is. $\underline{\alpha}_q(x)$ and $ \overline{\alpha}_q(x)$ are, respectively, the lower and upper growing rates of $C_q(x,n,\ve)$ as $n\to\infty$ and $\ve\to 0$ (in this order).

\begin{defi}[Energy function] 
Let $X$ be a general metric space and let $\mu$ be a Borel probability measure on $X$. For $q\in \R\setminus\{1\}$ and  $\ve\in(0,1)$, one defines the so-called energy function $I_{\cdot}(q,\ve):\M\rightarrow(0,+\infty]$ by the law
\begin{eqnarray}
\label{ef}
I_{\mu}(q,\ve)=\int_{\supp(\mu)}\mu(B(x,\ve))^{q-1}d\mu(x),
\end{eqnarray}
where $\supp(\mu)$ is the topological support of $\mu$.
\end{defi}

The next result shows that the two previous definitions are intimately related.

\begin{teo}[Pesin \cite{Pesin1993,Pesin1997}]
\label{Pesincorrelation}
Let $X$ be a Polish metric space, assume that $\mu$ is ergodic and let $q\in\N\setminus\{1\}$. Then, there exists a set $Z\subset X$ of full $\mu$-measure such that, for each $R,\eta>0$ and each $x\in Z$, there exists $N=N(x,\eta,R)\in\N$ such that 
\[|C_q(x,n,\varepsilon)-I_{\mu}(q,\varepsilon)|\leq \eta\]
holds for each $n\geq N$ and each $0<\varepsilon\leq R$. 
\end{teo}

\begin{defi}[Generalized fractal dimensions]\label{gfd}
 Let $X$ be a general metric space, let $\mu$ be a Borel probability measure on $X$, and let $q\in (0,\infty)\setminus\{1\}$. The so-called upper and lower $q$-generalized fractal dimensions of $\mu$ are defined, respectively, as
\[ D^+_{\mu}(q)=\limsup_{\varepsilon\downarrow 0} \frac{\log I_{\mu}(q,\varepsilon)}{(q-1)\log \varepsilon}  ~~\mbox{ and  }~~ D^-_{\mu}(q)=\liminf_{\varepsilon\downarrow 0} \frac{\log I_{\mu}(q,\varepsilon)}{(q-1)\log \varepsilon},\]
with values in $[0,+\infty]$.
For $q=1$, one defines the so-called upper and lower entropy dimensions, respectively, as
\[D^+_{\mu}(1)=\limsup_{\varepsilon\downarrow 0} \frac{\int_{\supp(\mu)} \log \mu(B(x,\ve))d\mu(x)}{\log \varepsilon},\]
\[D^-_{\mu}(1)=\liminf_{\varepsilon\downarrow 0} \frac{\int_{\supp(\mu)} \log \mu(B(x,\ve))d\mu(x)}{\log \varepsilon}.\]
\end{defi}

Some useful relations involving the generalized fractal, Hausdorff and packing dimensions of a probability measure are given by the following inequalities.

\begin{propo}[Proposition~1.2 in~\cite{AS2}]
\label{BGT1} 
Let $\mu$ be a Borel probability measure over $X$, let $q>1$ and let $0<s<1$. Then,
\begin{eqnarray*}
D^-_{\mu}(q)\leq \dim_H^-(\mu)\leq  \dim_H^+(\mu)\leq D^-_{\mu}(s), \\
D^+_{\mu}(q)\leq \dim_P^-(\mu)\leq  \dim_P^+(\mu)\leq D^+_{\mu}(s). 
\end{eqnarray*}
Furthermore, if $\supp(\mu)$ is compact, then $D_{\mu}^{\pm}(q)\leq D_{\mu}^{\pm}(1)\leq D_{\mu}^{\pm}(s)$.
\end{propo}

As in~\cite{AS}, we are also interested in the polynomial returning rates of the $T$-orbit of a given point to arbitrarily small neighborhoods of itself (which gives a quantitative description of  Poincar\'e's recurrence). This question was posed and studied by Barreira and Saussol in \cite{Barreiral2001} (see also~\cite{Hongwei,Barreira2002,Hu} for further motivations). Given a separable metric space $X$ and a Borel measurable transformation $T$, they have defined the lower and upper recurrence rates of $x\in X$ in the following way: for each fixed $r>0$, let
\begin{eqnarray*}
\tau_r(x)=\inf\{k\in\N\mid T^k x \in \overline{B}(x,r)\}
\end{eqnarray*}
be the  return time of a point $x \in X$ into the closed ball $\overline{B}(x, r)$; 
then,
\begin{eqnarray*}
\underline{R}(x)=\liminf_{r\to 0}\frac{\log \tau_r(x)}{-\log r} ~~\mbox{ and } ~~  \overline{R}(x)=\limsup_{r\to 0}\frac{\log \tau_r(x)}{-\log r}
\end{eqnarray*}
are, respectively, the lower and upper recurrence rates of $x \in X$. Note that $\tau_r(x)$ may be infinite on a set of zero $\mu$-measure.

We refer to~\cite{AS} for a discussion regarding the connection between these recurrence rates and the local dimensions of invariant measures.


We also have something to say  about the quantitative waiting time indicators, defined by Galatolo in~\cite{Galatolo} as follows: let $x,y\in X$ and let $r>0$. The first entrance time of $\mathcal{O}(x):=\{T^ix\mid i\in\mathbb{Z}\}$, the $T$-orbit of $x$, into the closed ball $\overline{B}(y,r)$ is given by
\begin{eqnarray*}
\tau_r(x,y)=\inf\{n\in \N \mid T^n(x)\in \overline{B}(y,r) \}
\end{eqnarray*} 
(note that $\tau_r(x,y)$ may be infinite on a set of zero $\mu\times\mu$-measure).

Naturally, $\tau_r(x,x)$ is just the first return time into the closed ball $\overline{B}(x,r)$. The so-called quantitative waiting time indicators are defined as
\begin{eqnarray*}
\underline{R}(x,y)=\liminf_{r\to 0}\frac{\log \tau_r(x,y)}{-\log r} ~~\mbox{ and } ~~  
\overline{R}(x,y)=\limsup_{r\to 0}\frac{\log \tau_r(x,y)}{-\log r}.
\end{eqnarray*}

Let $(X, T)$ be a dynamical system such that $X$ is a separable metric space and  $T : X \rightarrow X$ is a measurable map, and suppose that there exists a $T$-invariant measure $\mu$. Then, Theorem~4 in~\cite{Galatolo} states that, for each fixed $y\in X$,  one has
\begin{equation}\label{Gala} \underline{R}(x,y)\geq \underline{d}_{\mu}(y)\qquad
  \textrm{and}\qquad \overline{R}(x,y)\geq \overline{d}_{\mu}(y)\qquad \textrm{for}\;\;\mu\textrm{-a.e.}\,\;x\in X.
\end{equation}

Furthermore, even if $\mu$ is only a probability measure on $X$, Theorem~10 in~\cite{Galatolo} states that for each $x\in X$, one has $\underline{R}(x,y)\geq \underline{d}_{\mu}(y)$ and $\overline{R}(x,y)\geq \overline{d}_{\mu}(y)$ for $\mu$-a.e.$\;y\in X$.

\begin{defi}
\label{genericprop}
A subset $\mathcal{R}$ of a topological space $X$ is said to be residual if $\mathcal{R}\supset\bigcap_{k\in \N}U_k$, where for each $k\in\N$, $U_k$ is open and dense. A topological space $X$ is a Baire space if every residual subset of $X$ is dense in $X$ (by Baire Category Theorem, every complete metric space is a Baire space).

A property $\mathbb{P}$ is said to be generic in $X$ if there
exists a residual subset $\mathcal{R}$ of $X$ such that every element $x\in \mathcal{R}$ satisfies property $\mathbb{P}$.
\end{defi}



\subsection{Main results}

Our first result establishes that if $X=\prod_{-\infty}^\infty M$ is endowed with any metric compatible with the product topology for which $T$ is Lipshitz continuous (let, for instance, $d$ be given by~\eqref{metric1}), and if $M$ is any infinite Polish metric space (that is, $M$ may be a countable or an uncountable set), then the sets of invariant measures with (i) infinite packing dimension, (ii) infinite upper recurrence rate, for a.e. $x\in X$, and (iii) infinite upper quantitative time indicator for a.e. $(x,y)\in X\times X$, are residual subsets of $\M(T)$. 

\begin{teo}
\label{teocentral1}
Let $(X,T,\mathcal{B})$ be the full-shift dynamical system over $X=\prod_{-\infty}^{+\infty}M$, where the alphabet $M$ is any infinite Polish metric space and $X$ is endowed with any metric compatible with the product topology for which $T$ is Lipshitz continuous. Then, 
\begin{enumerate}
\item[1.] the set $PD:=\{ \mu\in \M(T) \mid \dim_{P}(\mu)= +\infty\}$ is a dense $G_{\delta}$ subset of $\mathcal{M}(T)$;
\item[2.] the set $\overline{\mathcal{R}}=\{ \mu\in \M(T) \mid \overline{R}(x)= +\infty,$ for $\mu\textrm{-}a.e.\,x\}$ is a dense $G_{\delta}$ subset of $\mathcal{M}(T)$;
\item[3.] the set $\overline{\mathscr{R}}=\{ \mu\in \M(T) \mid \overline{R}(x,y)= +\infty,$ for $(\mu\times\mu)\textrm{-}a.e.\,(x,y)\in X\times X\}$ is residual in $\mathcal{M}(T)$.
\end{enumerate}
\end{teo}

Thus, Theorem~\ref{teocentral1} extends the results stated in items IV, VI and VIII of Theorem~1.1 in~\cite{AS} to the case where $M$ is a \textit{countable} Polish metric space.

\begin{rek} It is also possible to show that if $M$ is any Polish metric space    and if $X=\prod_{-\infty}^{+\infty}M$ is endowed with any metric compatible with the product topology, then each one of the following sets is a dense $G_{\delta}$ subset of $\mathcal{M}(T)$:
  \begin{enumerate}
  \item $HD:=\{ \mu\in \M(T) \mid \dim_{H}(\mu)= 0\}$ 
    (this is, in fact, a consequence of the results presented in Section~2 in~\cite{AS});
  \item $\underline{\mathcal{R}}:=\{ \mu\in \M(T) \mid \underline{R}(x)=0,$ for $\mu\textrm{-}a.e.\,x\}$ 
    (this is Theorem~1.1-V in~\cite{AS});
  \item $\underline{\mathscr{R}}=\{ \mu\in \M(T) \mid \underline{R}(x,y)=0,$ for $(\mu\times\mu)\textrm{-}a.e.\,(x,y)\in X\times X\}$ 
    (this is Theorem~1.1-VII in~\cite{AS}).  
  \end{enumerate}
  
  Therefore, if $M$ is any infinite Polish metric space and if $X$ is endowed with any metric compatible with the product topology for which $T$ is Lipshitz continuous, it follows that
  \[
  HD\cap PD\cap\underline{\mathcal{R}}\cap\overline{\mathcal{R}}\cap\underline{\mathscr{R}}\cap\overline{\mathscr{R}}\] is a residual subset of $\mathcal{M}(T)$. 

  We refer to Introduction in~\cite{AS} for a discussion about the dynamical implications of this result.
  \end{rek}

The next result establishes that if the alphabet $M$ is a compact countable metric space and 
if $X$ is endowed with the metric given by~\eqref{metric2}, then generically for each $q\ge 1$, $\mu\in \M(T)$ has $q$-upper generalized fractal dimension equal to infinite. 

\begin{teo}
\label{teocentral2}
Let $(X,T,\mathcal{B})$ be the full-shift dynamical system over $X=\prod_{-\infty}^{+\infty}M$, where the alphabet $M$ is a countable compact set and $X$ is endowed with the metric given by~\eqref{metric2}. Then,  
\begin{enumerate}
\item[1.] for each $q>1$, 
  $CD:=\{ \mu\in \M(T) \mid D_{\mu}^+(q)=+\infty\}$  is a dense $G_\delta$ subset of $\mathcal{M}(T)$; 
\item[2.] $\{ \mu\in \M(T) \mid D_{\mu}^+(1)=+\infty\}$ is a residual subset of $\mathcal{M}(T)$.  
\end{enumerate}
\end{teo}

This partially settles the question posed in Remark~1.1 in~\cite{AS2}, i.e., that $CD$ is also a dense $G_\delta$ subset of $\mathcal{M}(T)$ in case $M$ is a countable compact metric space and $X$ is endowed with a sub-exponential metric (like~\eqref{metric2}; the proof of Proposition~\ref{central2} fails if the metric is exponential, as discussed in Remark~3.3 in~\cite{AS2}).  

\begin{rek} Theorem~\ref{teocentral2} is only valid for the case where $M$ is a countable compact metric space; this hypothesis is necessary for the proof of Proposition~\ref{central2}. Actually, one needs that $M$ has a countable collection of isolated points, with at most a finite number of accumulation points (given that $M$ is compact, $M$ contains at least one accumulation point; note that the set of accumulation points in $M$ cannot be infinite, otherwise $M$ would be uncountable).
\end{rek}






\begin{rek} It follows from Theorem~1.2 in~\cite{AS2} that if $M$ is any compact metric space  and if $X=\prod_{-\infty}^{+\infty}M$ is endowed with any metric compatible with the product topology, then for each $s\in(0,1)$,
  \[\{\mu\in \M(T) \mid D_\mu^-(s)= 0\}\] is a dense $G_{\delta}$ subset of $\mathcal{M}(T)$; this is, in fact, a consequence of the results presented in Section~2 of~\cite{AS}.

  Therefore, if $M$ is any  countable compact metric space and if $X$ is endowed with the metric~\eqref{metric2}, it follows from Proposition~\ref{BGT1} and Theorem~\ref{teocentral2} that for each $0<s\le 1$ and each $q\ge 1$,
  \[\{ \mu\in \M(T) \mid D_\mu^-(s)=0\;\textrm{and}\;D_\mu^+(q)=+\infty\}\] is a residual subset of $\mathcal{M}(T)$.

  We refer to the discussion after Theorem~1.3 in~\cite{AS2} for the dynamical consequences of this result.
  \end{rek}  

Finally, we have something to say about dynamical systems that are conjugated to full-shift systems. The idea here is to present a sufficient condition for a dynamical system to have a residual set of invariant measures whose packing and upper $q$-generalized ($q\ge 1$) dimensions are infinite, as well as the upper recurrence rate of a point and the upper quantitative waiting time indicator. Naturally, there are some important examples of dynamical systems which are conjugated to full-shift systems over infinite alphabets, like infinite interval exchange transformations (see~\cite{Iommi2013,Lopez2017} and references therein). Here, in order to preserve the positive dimensional and recurrence quantities discussed for the full-shift over countable alphabets, one has to require more than topological conjugacy, something that usually is very hard to obtain. Nevertheless, if such kind of conjugation is at hand, then the following results provide the aforementioned properties. 

\begin{teo}\label{correlationconj}
  Let $(X,T)$ be a topological dynamical system and let $(Y,\tilde{T})$ be the full-shift system over the alphabet $M$, with $M$ a countable compact metric space. Let $\varphi: X \rightarrow Y$ be an $\alpha$-Hölder continuous bijective map whose inverse map is continuous. If $\varphi$ is a conjugation between $(X,T)$ and $(Y,\widetilde{T})$, that is, if
\[\widetilde{T}=\varphi \circ T \circ \varphi^{-1},
\]
then for each $q>1$,
  \[\mathcal{D}^+_X(T):=\{\mu_X\in\mathcal{M}_X(T)\mid D_{\mu_X}^+(q)=+\infty\}\] is a dense $G_\delta$ subset of $\mathcal{M}_X(T)$. 
\end{teo}

\begin{teo}\label{packingconjC}
Let $(X,T)$ and $(Y,\tilde{T})$ be such that both $X$ and $Y$ are Polish metric spaces and both $T$ and $\tilde{T}$ are continuous. Let $\varphi: X \rightarrow Y$ be a locally $\alpha$-Hölder continuous bijective map whose inverse map is uniformly continuous. If $\varphi$ is a conjugation between $(X,T)$ and $(Y,\widetilde{T})$, then: 
\begin{enumerate}
  \item $PD_X(T):=\{\mu_x\in\mathcal{M}_X(T)\mid \dim_P(\mu_X)=+\infty\}$ is a  dense $G_\delta$ subset of $\mathcal{M}_X(T)$;
\item  $\overline{\mathcal{R}}(T):=\{\mu\in \M(T) \mid \overline{R}(x;T)= +\infty,$ for $\mu\textrm{-}a.e.\,x\}$ is a dense $G_\delta$ subset of $\mathcal{M}(T)$;
  \item $\overline{\mathscr{R}}(T)=\{\mu\in \M(T) \mid \overline{R}(x,y;T)= +\infty,$ for $(\mu\times\mu)\textrm{-}a.e.\,(x,y)\in X\times X\}$ is a residual subset of $\mathcal{M}(T)$.
\end{enumerate}
\end{teo}

\begin{rek} One may replace the hypothesis that the conjugation $\varphi: X\rightarrow Y$ is (locally) $\alpha$-Hölder continuous with uniformly continuous inverse in  Theorems~\ref{correlationconj} and~\ref{packingconjC} by the hypothesis that $\varphi$ is quasi-$\alpha$-Hölder continuous. A bijective transformation $\xi:X\rightarrow Y$ between two (Polish) metric spaces is said to be quasi-$\alpha$-Hölder continuous if for each $\varepsilon\in(0,\min\{1,\alpha\})$, there exists $\delta>0$ such that if $d_X(x,x^\prime)<\delta$, then
  \begin{equation}\label{quasiHolder}
    \left\vert\dfrac{d_Y(\xi(x),\xi(x^\prime))}{d_X(x,x^\prime)}-\alpha\right\vert<\varepsilon.
  \end{equation}
  $\xi$ is called locally quasi-$\alpha$-Hölder continuous if for each $x\in X$ and each $\varepsilon\in(0,1)$, there exists $\delta>0$ such that for each $x^\prime\in B(x,\delta)$,~\eqref{quasiHolder} follows.
  
  This definition generalizes the notion of quasi-Lipshitz transformation presented in~\cite{Yang2017}, which corresponds to the case $\alpha=1$. It is straightforward to show that if $\xi$ is quasi-$\alpha$-Hölder continuous, then $\xi^{-1}$ is quasi-$(1/\alpha)$-Hölder continuous, and so uniformly continuous.

  It is also possible to show that if $\xi:X\rightarrow Y$ is quasi-$\alpha$-Hölder continuous, then there exists a monotone non-decreasing function $\zeta:(0,+
  \infty)\rightarrow(0,+\infty)$ satisfying $\lim_{r\downarrow 0}\zeta(r)=0$ so that for each $x\in X$ and each suficiently small $r$, one has $\xi(B(x,r))\subset B(\xi(x),r^{\alpha-\zeta(r)})$ (see Lemma~2.4 in~\cite{Yang2017}).

  Therefore, the results stated in Section~\ref{section3} are valid (after minor adaptations in their proofs) if we replace the hypothesis that the conjugation $\varphi: X\rightarrow Y$ is locally $\alpha$-Hölder continuous with uniformly continuous inverse by the hypothesis that $\varphi$ is quasi-$\alpha$-Hölder continuous. If one wants to assume that $\varphi$ is only locally quasi-$\alpha$-Hölder continuous, then one also has to assume that $\varphi^{-1}$ is uniformly continuous. 
  \end{rek}

\begin{rek} Although it is possible to present (by following the results stated in Remarks~\ref{rgen},~\ref{Rpackingconj} and~\ref{rrec}) sufficient conditions for a conjugation to preserve generic sets of invariant measures with zero lower $s$-generalized fractal dimension ($0<s<1$), zero Haudorff dimension and zero lower recurrence rate and quantitative waiting time indicator, these results are not as important as those discussed in Theorems~\ref{correlationconj} and~\ref{packingconjC}, given that there are fairly general situations for which such sets are generic; roughly, the topological dynamical system $(X,T)$ must be such that $\overline{\M_p(T)}=\M(T)$, where $\M_p(T)$ stands for the set of the $T$-periodic measures (see~\cite{AS2,AS} for the proof of this statement).  
\end{rek}

The paper is organized as follows. In Section~\ref{section2}, we present several results used in the proof of Theorem~\ref{teocentral1}. They are adapted versions, for this setting, of some of the results presented in Sections~2 and~3 in \cite{AS}. Section~\ref{section3} is devoted to the proof of Theorem~\ref{teocentral2} (here, we also adapt two results presented in Section~3 in~\cite{AS2} for this setting). Finally, in Section~\ref{section3} we present the proofs of Theorems~\ref{correlationconj} and~\ref{packingconjC}. 




\section{$PD$, $\overline{\mathcal{R}}$ and $\overline{\mathscr{R}}$ are residual sets}
\label{section2}


We begin this section proving that the set $PD$ defined in the statement of Theorem~\ref{teocentral1} is dense in $\M(T)$. One needs two ingredients in this proof. The first one is the fact that for each $K>0$, the set of ergodic measures whose respective metric entropies are greater than $K$ is dense (this is Proposition~\ref{central1}). The second one is the fact that if $T$ is a Lipschitz continuous function, then the lower packing dimension of an ergodic measure is greater than its metric entropy times a constant (this is Lemma~\ref{dimpos01}).






The next result is due to Parthasarathy (Theorem~3.3 in~\cite{Parthasarathy1961}; see also Theorem~2 in~\cite{Oxtoby1963}). 
Let $\overline{\mathcal{M}}(T)$ denote the space of $T$-invariant measures on $\overline{X}=\prod_{-\infty}^{+\infty}\overline{M}$, endowed with the weak topology, where $\overline{M}$ is a compactification of $M$. 

\begin{lema}
\label{denseperiodic}
Let $\mu\in\overline{\M}(T)$ be such that $\mu(X)=1$ and let $\ve>0$. Then, there exists a (normalized) $T$-periodic measure $\nu\in \overline{\M}(T)$ such that $\nu\in B(\mu;\ve)$. 
\end{lema}

The next result is an extension of Theorem~2 in~\cite{Sigmund1971} and Proposition~2.4 in~\cite{AS} to the space $X=\prod_{-\infty}^{+\infty} M$, where $M$ is any infinite (countable or uncountable) Polish metric space  (Theorem~2 in~\cite{Sigmund1971} was proved for $M=\mathbb{R}$ and Proposition~2.4 in~\cite{AS} was proved in case $M$ is a perfect Polish metric space). 

\begin{propo}
\label{central1}
Let 
$K>0$. Then, $\{\mu\in \M_e\mid h_\mu(T)> K\}$ is a dense subset of $\M(T)$ (where $\M_e$ denotes the set of $T$-ergodic measures).  
\end{propo}
\begin{proof}
  Since, by Lemma~\ref{denseperiodic}, $\mathcal{N}:=\{\mu\in\overline{\M}_p(T)\mid \mu(X)=1\}$ is dense in  $\overline{\M}_X:=\{\mu\in\overline{\M}(T)\mid \mu(X)=1\}$, one just has to show that for each $\mu\in\mathcal{N}$ and each $\ve>0$, there exists a $\xi\in B(\mu;\ve)\cap\overline{\M}_e$ 
  such that $h_\xi(T)> K$ and $\xi(X)=1$.  The rest of the proof follows the same argument presented in the proof of Theorem~2 in~\cite{Sigmund1971}.

Set $s:=[e^{6K/\ve}]+1$ and note that $\ve\log s>6K$. Let $\rho$ be the Bernoulli shift whose states are $y_1,\ldots,y_s\in M$ and whose probabilities are given by the $s$-tuple $(1/s,\cdots,1/s)$ (since $M$ is infinite, given any $s\in\N$, there exist $y_1,\ldots,y_s\in M$ such that for each $1\le i\neq j\le s$, $y_i\neq y_j$), and let $\nu=(1-\delta)\mu+\delta\rho$, where $\delta=\ve/3$; clearly, $\nu(X)=1$ and $\nu\in B(\mu;\ve/2)$. 

We affirm that $h_\nu(T)>2K$. Namely, since $h_\nu(T)=(1-\delta)h_\mu(T)+\delta h_\rho(T)$ (see Theorem~8.1 in~\cite{Walters}), it follows that $h_\nu(T)=\delta h_\rho(T)=\delta\log s>2K$.

The last ingredient used in the proof is the fact that $\overline{\M}_e$ is entropy dense (given that $(X,T)$ satisfies the specification property; see~\cite{eizenberg1994,Kwietniak,Pfister2004,Sigmund1974}), that is, for each $\nu\in\overline{\M}(T)$, each $\eta>0$ and each $\lambda>0$, there exists $\xi\in B(\nu;\eta)\cap\overline{\M}_e$ such that $\vert h_\nu(T)-h_\xi(T)\vert<\lambda$. Thus, there exists $\xi\in B(\nu;\delta)\cap\overline{\M}_e$ such that $\vert h_\nu(T)-h_\xi(T)\vert<K$. Moreover, it follows from the proof of Theorem~B in~\cite{eizenberg1994} (see also Proposition~2.3 in~\cite{Pfister2004}) that one can choose $\xi$ such that $\xi(X)=1$, given that $\nu$ is supported on $X$.

Combining this with the previous result, it follows that $\xi\in B(\mu;\ve)\cap\overline{\M}_e$ is such that $\xi(X)=1$ and $h_\xi(T)>K$.

Now, since $\overline{\M}_{h,X}:=\{\zeta\in\overline{\M}_e\mid \zeta(X)=1 \mbox{ and } h_{\zeta}(T)> K\}$ is dense in $\overline{\M}_{X}$, there exists a sequence $(\xi_n)$, $\xi_n\in\overline{\M}_{h,X}$, converging weakly to $\xi$, with $\xi(X)=1$ and $h_{\xi}(T)> K$. Let, for each $n\in\N$, $\xi'_{n}\in \M(T)$ be the measure induced by $\xi_n \in \overline{\M}_X$. A standard argument shows that $\xi'_n\to \mu'$ in $\M(T)$ (see Proposition 6.1 in \cite{Oxtoby1963}).
\end{proof}

\begin{rek} The main difference between the proofs of Proposition~\ref{central1} and Proposition~2.4 in~\cite{AS} is the fact that in Proposition~\ref{central1}, we use the fact that $\overline{\M}_e$ is entropy dense (which is true whether $M$ is countable or not), where in Proposition~2.4 in~\cite{AS} one uses Theorem~2 in~\cite{Sigmund1971}, a result which is valid (after some minor adaptations) in case $M$ is a perfect Polish metric space. 
\end{rek}

The next result is an extension of Lemma~2.3 in~\cite{AS} in case $X$ is any Polish space.

\begin{lema}
\label{dimpos01}
Let $(X,T)$ be such that $X$ is a Polish metric space, $T$ is a Lipshitz function with constant $\Lambda>1$, and let $\mu\in\M_e$ (assume that $\M_e\neq\emptyset$). Then, $\dim_P^-(\mu)\geq \frac{h_{\mu}(T)}{\log \Lambda}$. 
\end{lema}

\begin{proof}
  The result is a consequence of the estimates presented in the proof of Lemma~2.3 in~\cite{AS} and Corollary~2.1 in~\cite{AS4}.
  
  Fix $x\in X$,  $n\geq1$ and $\ve>0$. Given $y\in B(x,\ve \Lambda^{-n})$, one has, for each $0\le i\le n$, $\rho(T^iy,T^ix)\le\Lambda^i\rho(x,y)\le\Lambda^{i-n}\ve<\ve$, which shows that $y\in B(x,n,\ve)=\{z\in X\mid\rho(T^iz,T^ix)<\ve,\;\;\forall\; 0\le i\le n-1\}$. Hence, for each $x\in X$ and each $\ve>0$,
\begin{eqnarray*}
\overline{d}_{\mu}(x)  \ge   \limsup_{n\to \infty} \frac{\log\mu(B(x,\ve \Lambda^{-n}))}{\log \ve \Lambda^{-n}}  
&\geq & \limsup_{n\to \infty} \frac{\log\mu(B(x,n,\ve))}{-n} ~\frac{1}{\frac{-\log\ve}{n}+\log \Lambda}\\ 
&\geq & \limsup_{n\to \infty} \frac{\log\mu(B(x,n,\ve))}{-n} \frac{1}{\log\Lambda};
\end{eqnarray*}
hence, it follows from Proposition~\ref{BGT} that
\begin{eqnarray*}
\label{entropylocal}
\dim_P^-(\mu)\ge\mu\textrm{-}\essinf\overline{d}_{\mu}(x) 
&\geq & \mu\textrm{-}\essinf\left\{\lim_{\ve\to 0} \liminf_{n\to \infty} \frac{\log\mu(B(x,n,\ve))}{-n}\right\}\frac{1}{\log\Lambda}\\
&=& \underline{h}_{\mu}^{loc}(T)\frac{1}{\log\Lambda}.
\end{eqnarray*}

The result follows now from Corollary~2.1 in~\cite{AS4}.
  \end{proof}

\begin{propo}[Proposition~2.4 in~\cite{AS}]
\label{densepack}
Let $(X,T,\mathcal{B})$ be as in the statement of Theorem~\ref{teocentral1} and let $L>0$. Then, $\{\mu\in \M_e\mid\dim_P^-(\mu)> L\}$ is a dense subset of $\mathcal{M}_e$.
\end{propo}

Now we turn our attention to the upper recurrence rates. The next result is an extension of Proposition~3.3 in~\cite{AS} to the case where $M$ is a countable Polish metric space. Its proof combines the proof of Proposition~3.3 with Remark~3.1 in~\cite{AS}, with the difference that we replace Lemma~2.3 in~\cite{AS} by Lemma~2.2.

\begin{propo}
\label{denseretpos}
Let $(X,T,\mathcal{B})$ be as in the statement of Theorem~\ref{teocentral1} and let $L>0$. Then, $\{\mu\in \M_e\mid \mu\textrm{-}\essinf \overline{R}(x)\ge L\}$ is a dense subset of $\mathcal{M}_e$.
\end{propo}

\begin{proof1}
\begin{itemize}
\item[1.]  Note that, by Proposition 2.1 in~\cite{AS} (see also Remark~2.1 in~\cite{AS}) and by Proposition~\ref{densepack}, $PD=\bigcap_{L\ge 1}\{\mu\in \M_e\mid\dim_P^-(\mu)> L\}$ is a countable intersection of dense $G_{\delta}$ subsets of $\M_e$. The result follows now from the fact that $\M_e$ is a dense $G_\delta$ subset of $\mathcal{M}(T)$ (see~\cite{Oxtoby1963,Parthasarathy1961}).
%
\item[2.] Given that $\overline{\mathcal{R}}=\bigcap_{L\ge 1}\{\mu\in \M(T)\mid \mu\textrm{-}\essinf \overline{R}(x)\ge L\}$, the result  follows from Proposition~\ref{denseretpos}, Proposition~3.1 and Theorem~1.1-(I) in~\cite{AS}, and from the fact that $\M_e$ is a dense $G_\delta$ subset of $\mathcal{M}(T)$. 
\item[3.] The result is a direct consequence of item 1 and the second inequality in~\eqref{Gala}.
\end{itemize}
\end{proof1}







\section{$CD$ is a dense $G_\delta$ set for each $q>1$}
\label{section3}

Here, we only prove that for each $q>1$, $CD=\{\mu\in \M(T)\mid D^+_{\mu}(q)=+\infty\}$ is a dense subset of $\mathcal{M}(T)$ in case $M$ is a countable compact metric space (this is Proposition~\ref{central2}). The fact that these sets are $G_\delta$ subsets of $\M(T)$ is proved in Proposition~2.4 in~\cite{AS2}.

We adapt the proof of Proposition~3.3 in~\cite{AS2} for this setting. The strategy involves a modified version of the energy function $I_{\mu}(q,\ve)=\int \mu(B(x,\ve))^{q-1}d\mu(x)$: for each $q>1$, each $\ve>0$, each $n\in\N$ and each $\mu\in\M(T)$, set 
\[I_{\mu}^{n}(q,\ve):=\int \mu(B^{n}(x,\ve))^{q-1}d\mu(x),\]
where $B^{n}(x,\ve)  := \cdots \times M \times \cdots \times M \times B_M(x_{-n},\ve)\times \cdots  \times B_M(x_{n},\ve)\times M\times \cdots \times M\times \cdots$, and $B_M(z,\ve):=\{w\in M\mid d(w,z)<\ve\}$.

The next results are extracted from~\cite{AS2}.

\begin{lema}[Lemma~3.1 in~\cite{AS2}]
\label{inc00}
Let $\ve>0$. Then, there exists $n_0\in \N$ such that for each $x\in X$, $B(x,\ve) \subseteq B^{n_0}(x,\ve)$.
\end{lema}

\begin{propo}[Proposition~3.2 in~\cite{AS2}]
\label{propcorrelation}
Let $q>1$. Then, 
\[D_{\mu}^{+}(q)= \limsup_{\ve \to 0} \frac{ \log I_{\mu}(q,\ve)}{(q-1)\log \ve}\geq D_{\mu,n_0}^+(q):=\limsup_{\ve \to 0} \frac{\log  I^{n_0}_{\mu}(q,\ve)}{(q-1)\log \ve},\]
where $n_0=n_0(\ve)$ is given by Lemma \ref{inc00}.
\end{propo}

The last ingredient needed in the proof is an adaptation of Lemma~1 in~\cite{Sigmund1970}. Here, the alphabet $M$ must be a compact metric space with a countable set of isolated points and with at most a finite number of accumulation points (such accumulation points can be replaced by isolated points in the proof, which we omit; see the proof of Lemma~3.2 in~\cite{AS2} for details). 
\begin{lema}
  \label{L6S2}
  Let $\mu\in \M(T)$ and let $U$ be an open basic (weak) neighborhood of $\mu$. Then, there exist $m_0,n_0\in\N$ such that for each $\tau\ge m_0n_0$, $\mu_x\in U\cap \mathcal{M}(T)$, where $x=(x_i)$ is a $T$-periodic point with period $\tau$ and $\mu_x(\cdot):=\frac{1}{\tau}\sum_{i=0}^{\tau-1}\delta_{T^ix}(\cdot)$. Moreover, there exists $\tilde{N}\in\N$ such that for each $N\in\N$, one may choose $\tau\ge m_0n_0$ such that $\tau/s>\tilde{N}$, where $1<s\le n_0N$ is the number of different entries of $(x_0,\ldots,x_{\tau-1})$, with each $x_i$ an isolated point of $M$.
\end{lema}

\begin{propo}
\label{central2}
Let $\mu\in \mathcal{M}(T)$, let $V\subset\M(T)$ be a (weak) neighborhood of $\mu$ and let $q>1$. Then, there exists $\rho\in V$ such that $D_{\rho}^+(q)=+\infty$.
\end{propo}
\begin{proof}
  The proof follows closely the proof of Proposition~3.3 in~\cite{AS2}, with the required adaptations. We present all the details.

  Let $\delta>0$ and set 
\[V=V_{\mu}(f_1,\cdots, f_d;\delta) =\left\{ \sigma \in \mathcal{M}\mid \left| \int f_jd\mu - \int f_jd\sigma  \right|<\delta,\, j=1,\ldots,d \right\},\]
where each $f_j\in C(M^{\Z})$ (this is the set of continuous real valued functions on $M^{\Z}$, endowed with the supremum norm). One can further assume that there exists $N\in\N$ such that, for each $j=1,\ldots,d$, one has $f_j(x)=f_j(y)$ if, for each $|i|\leq N$, $x_i=y_i$. Note that since $M$ is compact, functions of this type form a dense set in $C(M^{\Z})$.

Let $L=\sup\{|f_j(x)|\mid x\in M^{\Z},\, j=1,\ldots,d\}$, let $0<\kappa<1$ be such that
\begin{equation}
\label{kappaineq}
\begin{split}
\kappa< (8L)^{-1}\,2^{-(2N+1)}\,\delta,\\
2[1-(1-\kappa)^{2N+1}]<(8L)^{-1}\delta,
\end{split}
\end{equation}
and set $a:=\kappa^q+(1-\kappa)^q<1$ (note that for each $0<\kappa<1$, the mapping $(0,\infty)\ni q\mapsto \kappa^q+(1-\kappa)^q\in (0,\infty)$ is monotone decreasing, and since $f(1)=1$, it follows that $f(q)<1$ for each $q>1$).

It follows from Lemma~\ref{L6S2} that for each $\tilde{N}\in\N$ such that $\frac{1}{\tilde{N}}\log(\tilde{N})<-\frac{\log(a)}{q}$, there exists a $T$-periodic point $w=(w_i)\in M^{\Z}$ with period $\tau$ such that
\begin{equation}\label{tau}\log(a)+q\cdot\frac{s}{\tau}\log\left(\frac{\tau}{s}\right)<0
  \end{equation}
and
$\mu_w\in V_{\mu}(f_1,\cdots, f_d; \delta/2)$. 

Let $\{a_1,\ldots,a_s\}$ be the distinct entries of $(w_0,\ldots,w_{\tau-1})$, and for each $i=1,\ldots,s$, let $m_i$ be the number of times that $a_i$ occurs in $(w_0,\ldots,w_{\tau-1})$; naturally, $m_i\ge 1$ and
\[\sum_{i=1}^sm_i=\tau.\]

Let, for each $i=1,\ldots,s$, $\{y_i^j\}_{j=1}^{m_i}\subset M$ be such that $y_i^j\neq y_k^l$ and $a_i\neq y_i^j$ for each $i,k=1,\ldots,s$, $j,l=1,\ldots,m_i$, and let $\phi:M\rightarrow M$ be the map defined by the law
\begin{eqnarray*}
  \phi(y_i^j)&=&a_i, \qquad i=1,\ldots,s\quad j=1,\ldots,m_i,\\
    \phi(a_i)&=&y_i^1,\qquad i=1,\ldots,s,\\
    \phi(x)&=&x,  \qquad x\in M\setminus[\cup_{i=1}^s(\{a_i\}\cup\cup_{j=1}^{m_i}\{y_i^j\})].
\end{eqnarray*}

Finally, let $\varphi:X\to X$ be the map defined by the law
\[\varphi(\ldots,x_{-n},\ldots,x_n,\ldots):=(\ldots,\phi(x_{-n}),\ldots,\phi(x_n),\ldots).\]

It is clear that $\varphi$ is measurable. Furthermore, it maps cylinder sets centered at $\{y_1^{m_1},\ldots,y_s^{m_s}\}$  to cylinder sets centered at $\{a_1,\ldots,a_s\}$; that is,  $\varphi(C^n_y)=\varphi([-n; y_{i_{-n}},\ldots,y_{i_{n}}])=C^n_w=[-n; \phi(y_{i_{-n}}),\ldots,\phi(y_{i_{n}})]$, where $y_{i_k}\in\{y_1^{m_1},\ldots,y_s^{m_s}\}$.

Following the proof of Lemma 7 in~\cite{Sigmund1971},  one defines, for $\tau$ defined as before, a Markov chain $\rho$ whose states are $y_1^{1},\ldots,y_1^{m_1}$, $y_2^{1},\ldots,y_2^{m_2},\ldots, y_s^{1},\ldots, y_s^{m_s}$ (ordered so that $(\phi^{-1}(y_{j_0}^{i_0}),\ldots,\phi^{-1}(y_{j_{\tau-1}}^{i_{\tau-1}})=(w_0,\ldots,w_{\tau-1})$; set $(y_0,\ldots,y_{\tau-1}):=(y_{j_0}^{i_0},\ldots,y_{j_{\tau-1}}^{i_{\tau-1}})$ and $y:=(\ldots,y_{\tau-1},y_0,y_1,\ldots,y_{\tau-2},y_{\tau-1},y_0,\ldots)$), whose initial probabilities are given by the $\tau$-tuple  $(1/\tau,\cdots,1/\tau)$ and whose  transition probabilities are given by the $\tau\times \tau$-matrix $p_{ij}$, where
\begin{eqnarray*}
p_{\tau\,1}&=&1-\kappa,\\
p_{i\,i+1}&= & 1-\kappa ~\,\mbox{ for } i=1,\ldots, \tau-1,\\
p_{i\,j}&=& \frac{\kappa}{\tau-1} ~\mbox{ otherwise.}
\end{eqnarray*}

Let $\rho_*:=\rho\circ\varphi^{-1}$ (the pushforward of the invariant measure $\rho$). It is clear that 
\begin{eqnarray}
\label{pushforward}
\rho_*(C^n_w)=\rho_*([-n; w_{i_{-n}},\ldots,w_{i_{n}}])&=&\rho(\varphi^{-1}([-n; w_{i_{-n}},\ldots,w_{i_{n}}]))\nonumber\\
&=&\rho(\cup C^n_y),
\end{eqnarray}
where $w_{i_k}\in\{a_1,\ldots,a_s\}$, and $\rho_*(C^n_x)=0$, otherwise.

  One can show (see the proof of Lemma 7 in~\cite{Sigmund1971}) that $\rho_*\in V_{\mu_w}(f_1,\cdots, f_d; \delta/2)$, from which follows that $\rho_*\in V_{\mu}(f_1,\cdots, f_d; \delta)$.

Now, by Proposition \ref{propcorrelation}, one just needs to prove that $D_{\rho_*,n_0}^+(q)=\infty$.  Let $\ve\in(0,\min\{1,\ve_0\})$, with $\ve_0:=\min\{\min\{d(a_i,a_l) \mid \,i,l=1,\ldots,\tau, i\neq l\},\min\{d(a_i,z)\mid z\in M\setminus\{a_1,\ldots,a_s\}\}\}$ (naturally $\ve_0>0$, since each $a_i$ is an isolated point of $M$), and set $n=n_0(\ve)$. 

For each $x\in C^n_w$, it is clear from the choice of $\ve$ that $C^n_w= B^n(x,\ve)=\{ (z_i)_{i\in \Z}\in X\mid z_i\in B(x_i,\ve), ~i=-n,\ldots,n\}$ and so, from relation~\eqref{pushforward}, that
\[\rho_*(B^n(x,\ve))=\rho(\cdots\times M\times\phi^{-1}(\{x_{-n}\})\times \phi^{-1}(\{x_{n}\})\times M\times\cdots)=\rho(\cup C^n_y)\]
(if there exists $-n\le k\le n$ such that $x_k\notin\{a_1,\ldots,a_s\}$, then $\rho_*(B^n(x,\ve))=0$).

Moreover, as in Lemma 7 in \cite{Sigmund1971}, there exist $s^{2n+1}$ $C^n_w$-like sets (and also $\tau^{2n+1}$ $C^n_y$-like sets) that can be split into two groups, say $P$ and $Q$. $P$ consists of those $s$ sets  which contain an element of the orbit of $w$. The second group, $Q$, splits into the groups $Q_1,\ldots, Q_{2n}$, where $Q_p$ is the group of those $s \,\binom {2n} {p}\, (s-1)^p$ $C^n_w$-like sets for which there are exactly $p$ places $i=-n,\ldots,n$ where $x_{i+1}$  is not the {\em natural follower} of $w_i$, in the sense that if $x_i=w_l$ and $x_{i+1}=w_m$, then $m\neq l+1(\modd s)$.

Thus, since $I_{\rho_*}^n(q,\ve)$ depends only on the values taken by $\rho_*(B^{n}(x,\ve))$ when $x$ ranges over the $C^n_w$-like sets described above, one has 
\begin{eqnarray}
  \label{e7}
\int \rho_*(B^{n}(x,\ve))^{q-1}d\rho_*(x)\nonumber
&=& \int \rho_*(C^n_x)^{q-1}d\rho_*(x)\nonumber\\
&=& \sum_{j=1}^{s^{2n+1}} \int_{\{C^n_{w,j}\}} \rho_*(C^n_{w,j})^{q-1}d\rho_*(x)+   \int_{X\setminus\left\{C^n_{w,j}\right\}_{j=1}^{s^{2n+1}}} \rho_*(C^n_{w,j})^{q-1}d\rho_*(x)\nonumber\\
&=& \sum_{j=1}^{s^{2n+1}} \int_{\{C^n_{w,j}\}} \rho_*(C^n_{w,j})^{q-1}d\rho_*(x) = 
\sum_{j=1}^{s^{2n+1}} \rho_*(C^n_{w,j})^{q-1}\rho_*(C^n_{w,j})\nonumber\\
&=&\sum_{j=1}^{s^{2n+1}} \rho_*(C^n_{w,j})^{q}
~=~ \sum_{C^n_w\in P}\rho_*(C^n_w)^{q} + \sum_{p=1}^{2n}\sum_{C^n_w\in Q_p}\rho_*(C^n_{w})^{q},
\end{eqnarray}
where, in the second line, we have used that for each $x\in C^n_w$ and each $0<\ve<\ve_0$, $\rho_*(B^n(x,\ve))=\rho_*(C^n_w)$, as previously discussed. 

Note that given the $\tau$-uple $w_\tau^k:=(w_{k},\ldots,w_{\tau+k-1})$, with $k\in\Z$ (which is a block of size $\tau$ of the periodic point $w$), there exist $m_1\cdots m_s$ possible combinations of the entries of $w_\tau^k$ that also result in $w_\tau^k$; thus, since each $C_{w}^n\in P$ is centered at a translate of $w$, and so it is the adjoining of at most $[(2n+1)/\tau]+1$ $\tau$-uples of type $w_\tau^k$ for some $k\in\{0,\ldots,\tau-1\}$ (here, $[z]$ stands for the integer part of $z\in\mathbb{R}$; we also assume that $(2n+1)/\tau>1$), it follows that $N_P:=\#\{C_{y,j}^n\mid\varphi(C_{y,j}^n)=C_w^n\}$ (that is, the number of cylinders $C^n_{y,j}$ such that $\varphi(C^n_{y,j})=C_w^n$) is at most $A:=\left(\prod_{i=1}^sm_i\right)^{(2n+1)/\tau+1}$, for each $C_w^n\in P$; therefore, it follows that
\begin{eqnarray}\label{e71}
  \nonumber \sum_{C^n_w\in P}\rho_*(C^n_w)^{q}
  &\le& \sum_{C^n_w\in P}\left(\sum_{j:\varphi(C_{y,j}^n)=C_w^n}\rho(C_{y,j}^n)\right)^q\le s\cdot N_P^{q-1}\sum_{j=1}^{N_P}\rho(C_{y,j}^n)^q\\
 & \le & \frac{A^q}{\tau^{q-1}}(1-\kappa)^{2nq},
\end{eqnarray}
where we have used Jensen's inequality in the second inequality and the fact that for each $C^n_{y,j}=[-n;y_{i_{-n}},\ldots,y_{i_n}]$ such that $\varphi(C^n_{y,j})\in P$, 
\[\rho(C^n_{y,j})=\frac{1}{\tau}p_{y_{i_{-n}}y_{i_{-n+1}}}\cdots p_{y_{i_{n-1}}y_{i_{n}}}\le \frac{1}{\tau}\cdot(1-\kappa)^{2n}.\]
 
Now, fix $p\in\{1,\ldots,2n\}$, $C^n_w\in Q_p$, and let $\{k_l\}_{l=1}^p$, with $k_1<\cdots<k_p$, be the indices on the ($2n+1$)-uple $(x_{-n},\ldots,x_n)$ (at which $C^n_w$ is centered) such that for each $l\in\{1,\ldots,p\}$, $x_{k_l}$ is not the natural follower of $x_{k_{l-1}}$. Then, one has (by using the same reasoning as before)
\begin{eqnarray*}\rho_*(C^n_{w})^{q}&=&\left(\sum_{j:\varphi(C_{y,j}^n)=C^n_w}\rho(C_{y,j}^n)\right)^{q}\le  \frac{A^q}{\tau^{q}}\left(\sum_{i_1\in\{1,\ldots, s\}:i_1\neq k_1}m_{i_1}^q\right)\cdots\\
  &&\cdots\left(\sum_{i_p\in\{1,\ldots, s\}:i_p\neq k_p} m_{i_p}^q\right)  \left( \frac{\kappa}{\tau-1} \right)^{pq}(1-\kappa)^{(2n-p)q},
\end{eqnarray*}
and so
\begin{eqnarray}
\label{e8}
\sum_{p=1}^{2n}\sum_{C^n_w\in Q_p}\rho_*(C_{w}^n)^q&\le& s\cdot\frac{A^q}{\tau^{q}}\sum_{p=1}^{2n} \binom {2n} {p}\left((s-1)^{1-q}\kappa^q\right)^p((1-\kappa)^q)^{(2n-p)}\nonumber\\ 
&=& \frac{A^q}{\tau^{q-1}}\left( \left((s-1)^{1-q}\kappa^q+(1-\kappa)^q \right)^{2n}- (1-\kappa)^{2nq}\right),
\end{eqnarray}
where we have used the fact that for each $p\in\{1,\ldots,2n\}$ and each $p$-uple $(k_1,\ldots k_p)$,
\[\left(\sum_{i_1\in\{1,\ldots, s\}:i_1\neq k_1}m_{i_1}^q\right)\cdots\left(\sum_{i_p\in\{1,\ldots, s\}:i_p\neq k_p} m_{i_p}^q\right)\le (s-1)^p\left(\frac{\tau-1}{s-1}\right)^{pq}\]
(recall that $\sum_{i=1}^sm_i=\tau$, with $m_i\ge 1$), along with the fact that for each $C_{y,j}^n$ such that $\varphi(C_{y,j}^n)\in Q_P$,
\[\rho(C_{y,j}^n)\le \frac{1}{\tau}\cdot\left(\frac{\kappa}{\tau-1}\right)^{p}\cdot (1-\kappa)^{2n-p}.\]

Thus, by combining \eqref{e7} with \eqref{e71} and (\ref{e8}), one gets 
\begin{eqnarray*}
\int \rho_*(B^{n}(x,\ve))^{q-1}d\rho_*(x)\nonumber
&\le & \frac{A^q}{\tau^{q-1}}\left[ (1-\kappa)^{2nq} +  \left((s-1)^{1-q}\kappa^q+(1-\kappa)^q \right)^{2n}- (1-\kappa)^{2nq} \right]\\
&=&   \frac{A^q}{\tau^{q-1}} \left((s-1)^{1-q}\kappa^q+(1-\kappa)^q \right)^{2n}.
\end{eqnarray*}
Note that 
\[\log\left((s-1)^{1-q}\kappa^q+(1-\kappa)^q \right)<\log\left(\kappa^q+(1-\kappa)^q\right)=\log(a)<0\] (given that $s\ge 2$), from which follows that
\begin{eqnarray*}
\log\left(\int \rho_*(B^{n}(x,\ve))^{q-1}d\rho_*(x)\right)
&\le&
(q\log(A)-(q-1)\log(\tau)) +2n \log(a)\\
&=& \left(q\left(\frac{2n+1}{\tau} +1\right)\log\left(\prod_{i=1}^sm_i\right)-(q-1)\log(\tau)\right)+2n\log(a)\\
&\leq & \left(q\left(\frac{2n+1}{\tau} +1\right)s\log\left(\frac{\tau}{s}\right)-(q-1)\log(\tau)\right)+2n\log(a);
\end{eqnarray*}
here, we have used that $\prod_{i=1}^sm_i\le(\tau/s)^s$.

Recall that, by Lemma \ref{inc00}, one has~$n\ge (\frac{1}{\ve}-1)^{1/2}-1$; thus,
\begin{eqnarray}\label{desimp}
\nonumber \frac{\log\left(\int \rho_*(B^{n}(x,\ve))^{q-1}d\rho_*(x)\right)}{(q-1)\log \ve}
&\geq &  \frac{2}{q-1}\left(q\cdot\frac{s}{\tau}\log\left(\frac{\tau}{s}\right)+\log(a)\right)\frac{(1/\ve-1)^{1/2} }{\log \ve}+\frac{C_{q,\tau,s}}{\log(\ve)},
\end{eqnarray}
with $C_{q,\tau,s}$ a constant that does not dependent on $\ve$. By letting $\ve\to 0$, it follows from relation~\eqref{tau} that $D_{\rho_*,n}^+(q)=+\infty$.
\end{proof}

\begin{proof2}
\begin{enumerate}
\item[1.]
  The result follows from Proposition~2.4 in \cite{AS2}, Proposition~\ref{propcorrelation} and Proposition~\ref{central2}.
 \item[2.] It is a consequence of item 1 and Proposition~\ref{BGT1}. 
\end{enumerate}
\end{proof2}




\section{Lower packing, upper $q$-generalized fractal dimensions of invariant measures and upper recurrence rate, upper quantitative waiting time indicator under $\alpha$-Hölder conjugations}
\label{section4}

This section is devoted to the study of the fractal and generalized fractal dimensions of invariant measures, along with recurrence rates and quantitative waiting time indicators of $\alpha$-Hölder conjugated topological systems. %

The first result shows that given two topological dynamical systems $(X,T)$ and $(Y,\tilde{T})$, if there exists a conjugation between them that satisfies certain regularity conditions, then the upper $q$-generalized fractal dimensions ($q>1)$ of their invariant measures are directly related.

\begin{propo}
\label{correlation}
Let $(X,T)$ and $(Y,\tilde{T})$ be topological dynamical systems and let $\varphi: X \rightarrow Y$ be an $\alpha$-Hölder continuous (with constant $C>1$) bijective map with continuous inverse. If $\varphi$ is a conjugation between $(X,T)$ and $(Y,\widetilde{T})$, that is, if
\[\widetilde{T}=\varphi \circ T \circ \varphi^{-1},
\]
then there exists an open and continuous bijection  $ \psi: \M_Y(\widetilde{T}) \rightarrow \M_X(T)$ such that for each $q>1$ and each  $\nu_Y\in \M_Y(\tilde{T})$, 
\begin{eqnarray*}
D_{\psi(\nu_Y)}^+(q)\geq \alpha D_{\nu _Y}^+(q). 
\end{eqnarray*}
Furthermore, $\psi$ maps dense $G_\delta$ subsets of $\M_Y(\widetilde{T})$ into dense $G_\delta$ subsets of $\M_X(T)$.
\end{propo}
\begin{proof}
 Let $\M_X(T)$ and $\M_Y(\widetilde{T})$ denote, respectively, the spaces of $T$-invariant and $\widetilde{T}$-invariant measures, both endowed with the weak topology. Such topology is metrizable by the metric 
\[\vec{d}_X(\mu, \nu):=\inf \left\{\varepsilon>0 \mid \mu(A)<\nu\left(A^{\varepsilon}\right)+\varepsilon, \forall A \in \mathcal{B}(X)\right\},\]
where $A^{\varepsilon}:=\left\{x \mid d_X(x, A)<\varepsilon\right\}$ (see~\cite{gelfert} for a discussion).

Let the map $ \psi: \M_Y(\widetilde{T}) \rightarrow \M_X(T)$ be defined by the law
\[\psi\left(\nu_Y\right)(A):=\nu_Y(\varphi (A)),\qquad  A \in \mathcal{B}(X).\]

Firstly, we show that $\psi\left(\nu_Y\right) \in \M_X(T).$ Namely, one has for each $A\in\mathcal{B}(X)$,
\begin{eqnarray*}
\psi\left(\nu_Y\right)\left(T^{-1}(A)\right)&=& \nu_Y\left(\varphi\circ T^{-1}(A)\right)\\  
& =& \nu_Y\left(\varphi\circ T^{-1} \circ\varphi^{-1}\circ \varphi(A)\right)=\nu_Y\left(\tilde{T}^{-1}(\varphi(A))\right) \\
& =& \nu_Y(\varphi(A))=\psi\left(\nu_Y\right)(A),
\end{eqnarray*}
given that $\nu_Y\in \M_Y(\tilde{T})$ and $\varphi(A) \in \mathcal{B}(Y)$. Thus, $\psi\left(\nu_Y\right) \in \M_X(T)$.

\

\noindent{{\it{Claim.}}} $\psi$ is a bijection.

$\bullet$  $\psi$ is onto. Let $\nu_X \in \M_X(T)$ and set $\nu_Y :=\nu_X\circ \varphi^ {-1}$. We prove that $\nu_Y \in \M_Y(T)$. Let $\tilde{A}\in \mathcal{B}(Y)$; then,
\begin{eqnarray*}
\nu_Y \left(\widetilde{T}^{-1} (\tilde{A})\right)& =&\nu_Y\left((\widetilde{T}^{-1} \circ \varphi)(A)\right)=\nu_X\left((\varphi^{-1}\circ\widetilde{T}^{-1} \circ \varphi)(A)\right)\\
&=& \nu_X\left(T^{-1} (A)\right)= \nu_X (A)= \nu_Y\left(\varphi(A)\right)\\
&=& \nu_Y(\tilde{A}),
\end{eqnarray*}
where $A\in\mathcal{B}(X)$ is such that $\varphi(A)=\tilde{A}$. Thus, $\psi(\nu_Y)=\nu_X$ and $\psi$ is onto.

$\bullet$ $\psi$ is injective. Notice that the identity $\psi\left(\nu_Y\right)=\psi\left(\tilde{\nu}_Y\right)$ is equivalent, by definition, to $\nu_Y(\varphi(A))=\tilde{\nu}_Y(\varphi(A))$ for each $A \in \mathcal{B}(X)$, that is, to the identity $ \nu_X(A)=\tilde{\nu}_X(A)$.

So if $\nu_Y \neq \tilde{\nu}_Y$, then there exists $\tilde{A} \in \mathcal{B}(Y)$ such that 
$\nu_Y(\tilde{A}) \neq \tilde{\nu}_Y(\tilde{A})$. Thus,
$\psi\left(\nu_Y\right)(\varphi(A)) \neq \psi\left(\tilde{\nu}_Y\right)(\varphi(A))$, with $A\in \mathcal{B}(X)$ such that $\tilde{A}=\varphi(A)$. Hence, $\psi\left(\nu_Y\right)\neq \psi\left(\tilde{\nu}_Y\right)$, and $\psi$ in injective.


\

\noindent {\it Claim.} $\psi$ is continuous.

We must show that if $\vec{d}_Y\left(\mu_n, \nu\right) \rightarrow 0$, then $\vec{d}_X\left(\psi\left(\mu_n\right), \psi(\nu)\right) \rightarrow 0$.
We prove the counterpositive; suppose that there exist a subsequence $\left(\mu_{n_j}\right)$ and $\eta>0$ such that for each $j\in\mathbb{N}$, 
\[\vec{d}_X\left(\mu_{n_j} \circ \varphi, \nu \circ \varphi\right)= \inf \left\{\varepsilon>0 \mid (\mu_{n_j}\circ \varphi)(A)<(\nu \circ \varphi)\left(A^{\varepsilon}\right)+\varepsilon, ~\forall A \in \mathcal{B}(X)\right\}>\eta.\]
Then, there exists $A \in \mathcal{B}(X)$ such that for each $j\in\N$,
\begin{eqnarray}
\label{ineqcontinuous}
\left(\mu_{n_j} \circ \varphi\right)(A) \ge(\nu \circ \varphi)\left(A^{\eta}\right)+\eta.
\end{eqnarray}

Let us prove the following statement: for each $\eta>0$, there exists $\delta>0$ such that $\varphi(A)^\delta\subset \varphi\left(A^\eta\right)$. Namely, since $\varphi^{-1}: Y \rightarrow X$ is uniformly continuous (recall that both $X$ and $Y$ are compact), for each $\eta>0$ there exists $0<\delta<\eta$ such that if $d_Y(\varphi(x), \varphi(\tilde{x}))<\delta$, then 
$d_X(x, \tilde{x})<\eta$.

Now, if $y=\varphi(\tilde{x}) \in \varphi(A)^\delta$, then there exists $x \in A$ such that $d_Y(\varphi(\tilde{x}), \varphi(x))<\delta$ (namely, if
$y\in\varphi(A)^\delta=\left\{w \in Y \mid d_Y(w, \varphi(A))<\delta\right\}$, then $d_Y(y, \varphi(A))=\inf \left\{d_Y(y, \varphi(x)) \mid x \in A\right\}<\delta$, so there exists $x \in A$ such that $d_Y(y, \varphi(x))<\delta$), and then $d_X(\tilde{x}, x)<\eta$. Therefore,
$y \in \varphi\left(A^\eta\right)=\{\varphi(w) \mid w \in A^\eta\}=\left\{\varphi(w) \mid d_X\left(w, A\right)<\eta\right\}$, concluding the proof that $\varphi(A)^\delta\subset \varphi\left(A^\eta\right)$.

It follows from~\eqref{ineqcontinuous} and the statement above that for each $j\in\N$,
$\left(\mu_{n_j} \circ \varphi\right)(A) \ge (\nu \circ \varphi)\left(A^{\eta}\right)+\eta \ge \nu\left(\varphi(A)^{\delta}\right)+\delta$, that is, there exists $\tilde{A}\in \mathcal{B}(Y)$ such that for each $j\in \N$, $\mu_{n_j} (\tilde{A})  \ge \nu(\tilde{A}^{\delta})+\delta$, and so
\[\inf \left\{\varepsilon \mid \mu_{n_j}(B)<\nu (B^{\varepsilon})+\varepsilon, ~\forall B \in \mathcal{B}(Y)\right\}>\delta.\]

This shows that $\vec{d}_Y\left(\mu_n, \nu\right) \nrightarrow 0$.

\

\noindent {\it Claim.}  $\psi$ is an open map. 

It is sufficient to show that for each $\nu_Y\in\mathcal{M}_Y(\tilde{T})$, each $\eta>0$ and each 
$\nu_X \in \psi\left(B_Y\left(\nu_Y,\eta\right)\right)$, there exists   
$\delta>0$ such that $B_X\left(\nu_X, \delta\right) \subset \psi\left(B_Y\left(\nu_Y, \eta\right)\right)$.

Let $\nu_Y^{\prime}\in B_Y\left(\nu_Y,\eta\right)$ be such that $ \nu_Y^{\prime} \circ \varphi=\nu_X$, let $0<\delta<\min\left\{1,\left(\left(\eta-\vec{d}_Y(\nu_Y^{\prime}, \nu_Y)\right)\big/ 2 C\right)^{1 / \alpha}\right\}$, $\tilde{\nu}_X\in B_X\left(\nu_X , \delta\right)$, 
and let us assume for now that for each $A\in\mathcal{B}(X)$, $\varphi\left(A^\delta\right) \subseteq \varphi(A)^{C\delta ^\alpha}$; then, for each $\tilde{A}\in\mathcal{B}(Y)$, one has
\begin{eqnarray*}
\tilde{\nu}_Y(\tilde{A})=\tilde{\nu}_Y(\varphi(A)) &<& \nu_Y^{\prime}\left(\varphi\left(A^\delta\right)\right)+\delta \\
& <& \nu_Y^{\prime}\left(\varphi(A)^{C\delta^\alpha}\right)+C \delta^\alpha \\
& <& \nu_Y^{\prime}\left(\tilde{A}^{\left(\eta-\vec{d}_Y(\nu_Y^{\prime}, \nu_Y)\right)\big/ 2}\right)+\frac{\eta-\vec{d}_Y(\nu_Y^{\prime}, \nu_Y)}{2},
\end{eqnarray*}
given that $C \delta^\alpha<\left(\eta-\vec{d}_Y(\nu_Y^{\prime}, \nu_Y)\right)\big/ 2$ and $C \delta^\alpha>\delta$, where $A\in\mathcal{B}(X)$ is such that $\tilde{A}=\varphi(A)$ and $\tilde{\nu}_Y\in\mathcal{M}_Y(\tilde{T})$ is such that  $\tilde{\nu}_Y \circ \varphi=\tilde{\nu}_X$. 

Thus, $\vec{d}_Y\left(\tilde{\nu}_Y, \nu^{\prime}_Y\right)<\left(\eta-\vec{d}_Y(\nu_Y^{\prime}, \nu_Y)\right)\big/ 2$, from which follows that
\[\vec{d}_Y(\tilde{\nu}_Y, \nu_Y)\leq\vec{d}_Y\left(\tilde{\nu}_Y, \nu^{\prime}_Y\right)+\vec{d}_Y(\nu_Y^{\prime}, \nu_Y)<\frac{\eta+\vec{d}_Y(\nu_Y^{\prime}, \nu_Y)}{2}<\eta.\]

It remains to prove that for each $A\in\mathcal{B}(X)$, $\varphi\left(A^\delta\right) \subseteq \varphi(A)^{C\delta ^\alpha}$. Let $y \in \varphi\left(A^\delta\right)$; then,  there 
exists $w\in X$ such that $y=\varphi(w)$ and $d_X(w, A)<\delta$, and so there exists  $z\in A$ such that $d_X(w, z)<\delta$. Given that $\varphi$ is $\alpha$-Hölder continuous, it follows that 
$d_Y(\varphi(w), \varphi(z))<C \delta^{\alpha}$. 
Then, $d_Y(z, \varphi(A))<C \delta^\alpha$, that is, $y\in\varphi(A)^{C\delta ^\alpha}$.

\
 
\noindent {\it Claim.} $\psi$ maps $G_\delta$ dense subsets of $\M_Y (\widetilde{T})$ into $G_\delta$ dense subsets of $\M_X(T)$.

Since $\psi\left(\cap_n A_n\right)=\cap \psi\left(A_n\right)$ (if $x \in \psi\left(\cap A_n\right)$, then there exists $y \in \cap A_n$ such that $x=\psi(y)$, and so $x=\psi(y) \in \cap\psi\left( A_n\right)$; reciprocally, if $x \in \cap \psi\left(A_n\right)$, then there exists a sequence $\left(y_n\right)$ such that for each $n\in\N$, $x=\psi\left(y_n\right)$, and since $\psi$ is injective, one has $y=y_1\in \cap A_n$ and so $x \in \psi\left(\cap A_n\right)$), and $\psi$ is open, it follows that $\psi$ maps $G_\delta$ sets in $\M_Y(\tilde{T})$ into $G_\delta$ sets in $\M_X(T)$.

Now, since $\psi$ is continuous, if $\bar{\mathcal{A}}=\M_Y(\tilde{T})$, then $\overline{\psi(\mathcal{A})}=\M_X(T)$, so $\psi$ maps dense subsets of $\M_Y(\tilde{T})$ into dense subsets of $\M_X(T)$.

\


\noindent {\it Claim.} Let $q>1$ and $\nu_Y\in \M_Y(\tilde{T})$. Then, 
$D_{\psi(\nu_Y)}^+(q) \ge \alpha D_{\nu _Y}^+(q)$.

Since $\varphi$ is $\alpha$-Hölder, it follows that for each $x,x^\prime\in X$, $d_Y\left(\varphi(x), \varphi\left(x'\right)\right) \leq C d_X\left(x, x'\right)^\alpha$, and so for each $0<\varepsilon<1$ and each $\nu_Y\in\mathcal{M}_Y(\tilde{T})$, one has $\varphi(B(x , \varepsilon)) \subset B\left(\varphi(x) , C \varepsilon^\alpha\right)$ and $\psi\left(\nu_Y\right)(B(x , \varepsilon)) \leq \nu_Y\left(B\left(\varphi(x) , C \varepsilon^\alpha\right)\right)$. Therefore, for each $q>1$ and each $0<\varepsilon<1$ such that $C\varepsilon^\alpha<1$, one has 
\begin{eqnarray*}
\frac{\ln \int \nu_Y\left(B\left(y , C \varepsilon^\alpha\right)\right)^{q-1}d\nu_Y\left(y\right)}{(q-1)\ln \left(C \varepsilon^\alpha\right)}  \cdot   \frac{\ln C+\alpha \ln \varepsilon}{\ln \varepsilon} &= & \frac{\ln \int \nu_Y\left(B\left(y , C \varepsilon^\alpha\right)\right)^{q-1}d\nu_Y\left(y\right)}{(q-1)\ln \varepsilon}\\
&\le &\frac{\ln \int \nu_Y\left(\varphi( B\left(x ,  \varepsilon\right))\right)^{q-1}d\nu_Y\left(\varphi(x)\right)}{(q-1)\ln \varepsilon}\\
&=&\frac{\ln \int \psi\left(\nu_Y\right)(B(x , \varepsilon))^{q-1}d \psi\left(\nu_Y\right)(x)}{(q-1)\ln \varepsilon}.
\end{eqnarray*}
and thus $D_{\psi(\nu_Y)}^+(q)\geq \alpha D_{\nu _Y}^+(q)$.
\end{proof}

\begin{proof3}
Let $q>1$. The result is a consequence of Theorem~\ref{teocentral2} and Proposition~\ref{correlation}. Namely, it follows from Theorem~\ref{teocentral2} that $\mathcal{D}_Y^+(\tilde{T})$ is a dense $G_\delta$ subset of $\mathcal{M}_Y(\tilde{T})$. Now, since the map $\psi: \M_Y(\widetilde{T}) \rightarrow \M_X(T)$ is a bijection, it follows from the inequality $D_{\psi(\nu_Y)}^+(q)\geq \alpha D_{\nu _Y}^+(q)$ that
  \begin{eqnarray*}
    \mathcal{D}_X^+(T)&=&\bigcap_{L\ge 1}\{\psi(\nu_Y)\in\mathcal{M}_X(T)\mid D_{\psi(\nu_Y)}^+(q)>\alpha L\}\supset\psi\left(\bigcap_{L\ge 1}\{\nu_Y\in\mathcal{M}_Y(\tilde{T})\mid D_{\nu_Y}^+(q)> L\}\right)\\
    &=&\psi(\mathcal{D}_Y^+(\tilde{T})).
    \end{eqnarray*}

The result is now a consequence of the fact that $\psi$ maps dense $G_\delta$ subsets of $\mathcal{M}_Y(\tilde{T})$ into dense $G_\delta$ subsets of $\mathcal{M}_X(T)$.
\end{proof3}

\begin{rek} The hypothesis that $\varphi:X\rightarrow Y$ is $\alpha$-Hölder continuous is used in the proof that $\psi:\M_Y(\tilde{T})\rightarrow\M_X(T)$ is an open map and, naturally, in the proof that $D_{\psi(\nu_Y)}^+(q) \ge \alpha D_{\nu _Y}^+(q)$, for each $q>1$ and each $\nu_Y\in\M_Y(\tilde{T})$.

  The hypothesis that $\varphi^{-1}$ is (uniformly) continuous is needed in the proof that $\psi$ is continuous. 
\end{rek}  

\begin{rek}\label{rgen}
  Under the same hypotheses of Theorem~\ref{correlation}, the map 
    $\psi: \M_Y(\widetilde{T}) \rightarrow \M_X(T)$ is such that for each $0<s<1$ and each $\nu_Y\in\mathcal{M}_Y(\tilde{T})$, 
\begin{eqnarray*}
D_{\psi(\nu_Y)}^-(s)\leq \alpha D_{\nu _Y}^-(s). 
\end{eqnarray*}
Therefore, if $(X,T)$ and $(Y,\tilde{T})$ are conjugated topological dynamical systems, with $\overline{\mathcal{M}_p(\tilde{T})}=\mathcal{M}_Y(\tilde{T})$  
and whose conjugation $\varphi:X\to Y$ satisfies the hypotheses stated in Theorem~\ref{correlation}, then for each $0<s<1$, the set
  \[\{\mu_X\in\mathcal{M}_X(T)\mid D_{\mu_X}^-(s)=0\}\] is a dense $G_\delta$ subset of $\mathcal{M}_X(T)$. This result is a direct consequence of Theorem~1.2 in~\cite{AS2}.
\end{rek}

Now, we let $(X,T)$ and $(Y,\tilde{T})$ be such that both $X$ and $Y$ are Polish metric spaces. In this setting,  the existence of a conjugation $\varphi: X \rightarrow Y$ between $T$ and $\tilde{T}$ that is locally $\alpha$-Hölder continuous guarantees that to each $\mu_X\in\mathcal{M}_X(T)$ of positive lower packing dimension corresponds an invariant measure $\nu_Y\in\mathcal{M}_Y(\tilde{T})$ of positive lower packing dimension.

\begin{teo}\label{packingconj}
Let $(X,T)$ and $(Y,\tilde{T})$ be such that both $X$ and $Y$ are Polish metric spaces and both $T$ and $\tilde{T}$ are continuous. Let $\varphi: X \rightarrow Y$ be a locally $\alpha$-Hölder continuous bijective map whose inverse map is uniformly continuous. If $\varphi$ is a conjugation between $(X,T)$ and $(Y,\widetilde{T})$, 
then there exists an open and continuous bijection  $ \psi: \M_Y(\widetilde{T}) \rightarrow \M_X(T)$ such that for each $\nu_Y\in\mathcal{M}_Y(\tilde{T})$,
\begin{eqnarray*}
\dim_P^-(\psi\left(\nu_Y\right))\ge \alpha \dim_P^-(\nu_Y).
\end{eqnarray*}
Furthermore, $\psi$ maps dense subsets of $\M_Y(\widetilde{T})$ into dense subsets of $\M_X(T)$.
\end{teo}

\begin{proof}
The existence of the bijective and continuous map $\psi: \M_Y(\widetilde{T}) \rightarrow \M_X(T)$  follows the same arguments presented in the proof of Theorem~\ref{correlation}. Since $\psi$ is continuous, it maps dense subsets of $\mathcal{M}_Y(\tilde{T})$ into dense subsets of $\mathcal{M}_X(T)$.

It remains to prove that if $\nu_Y\in \M_Y(T)$, then $\operatorname{dim}_P^{-}(\psi\left(\nu_Y\right))\ge\alpha\operatorname{dim}_P^{-}(\nu_Y)$. 
By Proposition~\ref{BGT}, it is sufficient to prove that 
$\psi\left(\nu_Y\right)\textrm{-}\essinf d_{\psi\left(\nu_Y\right)}^+\ge\alpha\nu_Y\textrm{-}\essinf d_{\nu_Y}^{+}$.

Given that $\varphi$ is locally $\alpha$-Hölder continuous, it follows that for each $x\in X$, there exist $0<\varepsilon<1$ and a constant $C>1$ such that if $x^\prime\in B(x,\varepsilon)$, then 
$d_Y\left(\varphi(x), \varphi\left(x'\right)\right) \leq C d_X\left(x, x'\right)^\alpha$. Therefore, one has for each $x\in X$  and each $0<\varepsilon<1$ such that $C\varepsilon^\alpha<1$, the set inclusions $\varphi(B(x , \varepsilon)) \subset B\left(\varphi(x) , C \varepsilon^\alpha\right)$ and $\psi\left(\nu_Y\right)(B(x , \varepsilon)) \leq \nu_Y\left(B\left(\varphi(x) , C \varepsilon^\alpha\right)\right)$, and so 
\begin{eqnarray*}
\frac{\ln \psi\left(\nu_Y\right)(B(x , \varepsilon))}{\ln \varepsilon} &\ge & \frac{\ln \nu_Y\left(B\left(\varphi(x) , C \varepsilon^\alpha\right)\right)}{\ln \varepsilon}\\
&=&\frac{\ln \nu_Y\left(B\left(\varphi(x), C \varepsilon^\alpha\right)\right)}{\ln \left(C \varepsilon^\alpha\right)}  \cdot   \frac{\ln C+\alpha \ln \varepsilon}{\ln \varepsilon},
\end{eqnarray*}
from which follows that for each $x\in X$, $d_{\psi\left(\nu_Y\right)}^{+}(x) \geqslant \alpha d_{\nu_Y}^{+}(\varphi(x))$. This proves the inequality $\psi\left(\nu_Y\right)\textrm{-}\essinf d_{\psi\left(\nu_Y\right)}^+\ge\alpha\nu_Y\textrm{-}\essinf d_{\nu_Y}^{+}$.
\end{proof}

\begin{rek}\label{Rpackingconj}
 In the same setting of Proposition~\ref{packingconj}, if one assumes that the conjugation $\varphi:X\rightarrow Y$ is a continuous map with its inverse a locally $\alpha$-Hölder continuous map, then there exists an open and continuous bijection
 $\phi: \M_Y(\widetilde{T}) \rightarrow \M_X(T)$ such that for each  $\nu_Y\in\mathcal{M}_Y(\tilde{T})$, 
\[\alpha\dim_H^+(\phi(\nu_Y))\leq \dim_H^+(\nu _Y).\] 
The result is a consequence of the following modification of the proof of Proposition~\ref{packingconj}: if $\nu_Y\in \M_Y(\tilde{T})$, then $\nu_Y\textrm{-}\esssup d_{\nu_Y}^-\ge\alpha\phi(\nu_Y)\textrm{-}\esssup d_{\phi(\nu_Y)}^{-}$.

Namely, since $\varphi^{-1}$ is $\alpha$-Hölder, it follows that for each $y\in Y$, there exists $0<\ve<1$ and a constant $C>1$ such that if $y^\prime\in B(y,\ve)$, then $d_X\left(\varphi^{-1}(y), \varphi^{-1}\left(y'\right)\right) \leq C d_Y\left(y, y'\right)^\alpha$, and so for each $0<\varepsilon<1$ and each $\nu_Y\in\mathcal{M}_Y(\tilde{T})$, one has $\varphi^{-1}(B(y , \varepsilon)) \subset B\left(\varphi^{-1}(y) , C \varepsilon^\alpha\right)$ and $\phi\left(\nu_Y\right)(B(\varphi^{-1}(y), C\varepsilon^\alpha)) \geq \nu_Y\left(B\left(y, \varepsilon\right)\right)$. Thus, one has for each $y\in Y$ and each $0<\varepsilon<1$ such that $C\varepsilon^\alpha<1$, 
\begin{eqnarray*}
\frac{\ln\left( \nu_Y(B(y , \varepsilon))\right)}{\ln \varepsilon} &\ge & \frac{\ln\left( \phi(\nu_Y)\left(B\left(\varphi^{-1}(y) , C \varepsilon^\alpha\right)\right)\right)}{\ln \varepsilon}\\
&=&\frac{\ln \left(\phi(\nu_Y)\left(B\left(\varphi^{-1}(y), C \varepsilon^\alpha\right)\right)\right)}{\ln \left(C \varepsilon^\alpha\right)}  \cdot   \frac{\ln C+\alpha \ln \varepsilon}{\ln \varepsilon},
\end{eqnarray*}
from which follows that for each $y\in Y$, $d_{\nu_Y}^{-}(x) \geqslant \alpha d_{\phi(\nu_Y)}^{-}(\varphi^{-1}(y))$, and so $\nu_Y\textrm{-}\esssup d_{\nu_Y}^-\ge\alpha\phi(\nu_Y)\textrm{-}\esssup d_{\phi(\nu_Y)}^{-}$.

Therefore, if $(X,T)$ and $(Y,\tilde{T})$ are conjugated topological dynamical systems, with $\overline{\mathcal{M}_p(\tilde{T})}=\mathcal{M}_Y(\tilde{T})$  
and whose conjugation $\varphi:X\to Y$ satisfies the hypotheses stated above, then the set
  \[\{\mu_X\in\mathcal{M}_X(T)\mid \dim_H(\mu_X)=0\}\] is a residual subset of $\mathcal{M}_X(T)$. This result is a direct consequence of Corollary~1.1 in~\cite{AS2}.
\end{rek}

Finally, let us consider how the upper recurrence rate of $x\in X$ and the upper quantitative waiting time indicator behave under a locally $\alpha$-Hölder conjugation between $(X,T)$ and $(Y,\tilde{T})$.

\begin{propo}\label{recurrenceconj} Let $(X,T)$, $(Y,\tilde{T})$ and  $\varphi: X \rightarrow Y$ be as in the statement of Theorem~\ref{packingconj}. Then, 
  for each $x,x^\prime\in X$, 
\begin{eqnarray*}
\alpha\overline{R}(\varphi(x);\tilde{T})\le \overline{R}(x;T), \qquad \alpha\overline{R}(\varphi(x),\varphi(x^\prime);\tilde{T})\le \overline{R}(x,x^\prime;T).
\end{eqnarray*}
\end{propo}
\begin{proof} We just present the proof of the first inequality. Note that for each $k\in\N$ and each $x\in X$, one has $\tilde{T}^k(\varphi(x))=\varphi(T^k(x))$. Moreover, since $\varphi$ is locally $\alpha$-Hölder continuous, it follows that for each $x\in X$, there exist $0<r<1$ and a positive constant $C$ such that 
  $\varphi(\overline{B}(x,r)) \subset \overline{B}\left(\varphi(x),Cr^\alpha\right)$.

  By combining these statements, one gets for each $x\in X$ and each $0<r<1$,
  \begin{eqnarray*}
    \tau_{Cr^{\alpha}}(\varphi(x);\tilde{T})&=&\inf\{k\in\N\mid \tilde{T}^k(\varphi(x))\in \overline{B}\left(\varphi(x),Cr^\alpha\right)\}\\
    &\le&\inf\{k\in\N\mid \varphi(T^k(x))\in \varphi(\overline{B}(x,r))\}=\tau_{r}(x;T),
 \end{eqnarray*}   
where we have used the fact that for each $x\in X$ and each $k\in \N$, $T^k(x)\in \overline{B}(x,r)$ if, and only if, $\varphi(T^k(x))\in \varphi(\overline{B}(x,r))$. Thus, one has for each $x\in X$ and each $0<r<1$ such that $Cr^\alpha<1$,
\begin{eqnarray*}
  \frac{\ln \tau_{Cr^{\alpha}}(\varphi(x);\tilde{T})}{-\ln Cr^\alpha} \cdot   \frac{\ln C+\alpha \ln r}{\ln r}\le \frac{\ln \tau_r(x;T)}{-\ln r},
\end{eqnarray*}
from which the result follows.
  \end{proof}

\begin{rek}\label{rrec}
 In the same setting of Proposition~\ref{recurrenceconj}, if one supposes that the conjugation $\varphi:X\rightarrow Y$ is a continuous map with its inverse a locally $\alpha$-Hölder continuous map, then then there exists an open and continuous bijection
 $\phi: \M_Y(\widetilde{T}) \rightarrow \M_X(T)$ is such that for each 
 $\nu_Y\in\mathcal{M}_Y(\tilde{T})$ and each $x,x^\prime\in X$, 
\begin{eqnarray*}
\underline{R}(x;T)\leq \alpha \underline{R}(\varphi(x);\tilde{T}),\qquad \underline{R}(x,x^\prime;T)\leq \alpha \underline{R}(\varphi(x),\varphi(x^\prime);\tilde{T}).
\end{eqnarray*}
%
The result follows from the proof of Proposition~\ref{recurrenceconj}, but since the argument is similar to the one presented in Remark~\ref{Rpackingconj}, we do not show the details.
\end{rek}

\begin{proof4}  
\begin{enumerate} \item The result follows from Theorem~\ref{teocentral1}-(1) and Proposition~\ref{packingconj}. Namely, it follows from Theorem~\ref{teocentral1}-(1) that $PD_Y(\tilde{T})$ is a dense $G_\delta$ subset of $\mathcal{M}_Y(\tilde{T})$. Now, since the map $\psi: \M_Y(\widetilde{T}) \rightarrow \M_X(T)$ is a bijection, it follows from the inequality $\dim_P^-(\psi(\nu_Y))\geq \alpha \dim_P^-(\nu _Y)$ that for each $L\in\N$,
  \begin{eqnarray*}
    \{\psi(\nu_Y)\in\mathcal{M}_X(T)\mid \dim_P^-(\psi(\nu_Y))\ge\alpha L\}\supset\psi\left(\{\nu_Y\in\mathcal{M}_Y(\tilde{T})\mid \dim_P^-(\nu_Y)\ge L\}\right).
    \end{eqnarray*}

  The result is now a consequence of the following facts:
  \begin{itemize}
  \item $\psi$ maps dense subsets of $\mathcal{M}_Y(\tilde{T})$ into dense subsets of $\mathcal{M}_X(T)$ (and so, for each $L\in\N$, $\{\mu_X\in\mathcal{M}_X(T)\mid \dim_P^-(\mu_X)\ge\alpha L\}$ is a dense subset of $\mathcal{M}_X(T)$);
  \item for each $L\in\N$, $\{\mu_X\in\mathcal{M}_X(T)\mid \dim_P^-(\mu_X)\ge\alpha L\}$ is a $G_\delta$ subset of $\mathcal{M}_X(T)$, by Proposition~2.1 in~\cite{AS} (here we also use the fact that $\mathcal{M}_X(T)$ is a dense $G_\delta$ subset of $\mathcal{M}(X)$, the set of probability measures defined in $(X,\mathcal{B})$, endowed with the weak topology; see~\cite{Parthasarathy1961});
  \item $PD_X(T)=\bigcap_{L\ge 1}\{\mu_X\in\mathcal{M}_X(T)\mid \dim_P^-(\mu_X)\ge\alpha L\}$.
    \end{itemize}

\item  The result follows from Theorem~\ref{teocentral1}-(2) and Proposition~\ref{recurrenceconj}. Namely, it follows from Theorem~\ref{teocentral1}-(2) that $\overline{\mathcal{R}}_Y(\tilde{T})$ is a dense $G_\delta$ subset of $\mathcal{M}_Y(\tilde{T})$. Now,  
  it follows from the inequality $\alpha\overline{R}(\varphi(x);\tilde{T})\le \overline{R}(x;T)$, valid for each $x\in X$, that for each $L\in\N$,
  \begin{eqnarray*}
    \{\psi(\nu_Y)\in\mathcal{M}_X(T)\mid \overline{R}(x;T)\ge\alpha L\;\textrm{for}\;\psi(\nu_Y)\textrm{-a.e.}\,x\}&&\\
    \supset\psi\left(\{\nu_Y\in\mathcal{M}_Y(\tilde{T})\mid \overline{R}(\varphi(x);\tilde{T})\ge L\;\textrm{for}\;\nu_Y\textrm{-a.e.}\,\varphi(x)\}\right).
    \end{eqnarray*}

  The result is now a consequence of the following facts:
  \begin{itemize}
  \item 
    for each $L\in\N$, $\{\mu_X\in\mathcal{M}_X(T)\mid \overline{R}(x;T)\ge\alpha L\;\textrm{for}\;\mu_X\textrm{-a.e.}\,x\}$ is a dense ($\psi$ maps dense subsets of $\mathcal{M}_Y(\tilde{T})$ into dense subsets of $\mathcal{M}_X(T)$) $G_\delta$ (by Proposition~3.1 in~\cite{AS}) subset of $\mathcal{M}_X(T)$;
  \item $\overline{\mathcal{R}}(T)=\bigcap_{L\ge 1}\{\mu_X\in\mathcal{M}_X(T)\mid \overline{R}(x;T)\ge\alpha L\;\textrm{for}\;\mu_X\textrm{-a.e.}\,x\}$.
\end{itemize}
\item  It is a consequence of item 1 and the second inequality in~\eqref{Gala}.
\end{enumerate}
\end{proof4}


\bibliography{refs}{}
\bibliographystyle{acm}

\end{document}